\documentclass[
11pt,  reqno]{amsart}
\usepackage[margin=1.2in,marginparwidth=1.5cm, marginparsep=0.5cm]{geometry}

\usepackage{amsmath,amssymb,amscd,amsthm,amsxtra, esint}

\usepackage{booktabs} 
\usepackage{microtype}
\usepackage{amssymb}
\usepackage{mathrsfs}

\usepackage{upgreek} 

\usepackage{color}
\usepackage[implicit=true]{hyperref}

\usepackage{xsavebox}

\usepackage{cases}

\allowdisplaybreaks[2]

\definecolor{gr}{rgb}   {0.,   0.69,   0.23 }
\definecolor{bl}{rgb}   {0.,   0.5,   1. }
\definecolor{mg}{rgb}   {0.85,  0.,    0.85}
\definecolor{yl}{rgb}   {0.8,  0.7,   0.}
\definecolor{or}{rgb}  {0.7,0.2,0.2}

\newcommand{\D}{\mathcal{D}}

\newcommand{\Pt}[1]
{\mathcal{P}_{#1}}

\newcommand{\rhoo}{\vec{\rho}}
\newcommand{\muu}{\vec{\mu}}

\newtheorem{theorem}{Theorem} [section]

\newtheorem{lemma}[theorem]{Lemma}
\newtheorem{proposition}[theorem]{Proposition}
\newtheorem{remark}[theorem]{Remark}

\newtheorem{corollary}[theorem]{Corollary}

\DeclareMathOperator*{\intt}{\int}
\DeclareMathOperator*{\iintt}{\iint}
\DeclareMathOperator*{\supp}{supp}

\DeclareMathOperator{\Law}{Law}
\DeclareMathOperator{\Id}{Id}

\newcommand{\II}{\text{I \hspace{-2.8mm} I} }

\newcommand{\W}{\mathcal{W}}

\newcommand{\noi}{\noindent}
\newcommand{\Z}{\mathbb{Z}}
\newcommand{\R}{\mathbb{R}}

\newcommand{\T}{\mathbb{T}}
\newcommand{\N}{\mathbb{N}}

\let\Re=\undefined\DeclareMathOperator*{\Re}{Re}
\let\Im=\undefined\DeclareMathOperator*{\Im}{Im}

\let\P= \undefined
\newcommand{\P}{\mathbf{P}}
\newcommand{\PP}{\mathbb{P}}

\newcommand{\F}{\mathcal{F}}
\newcommand{\NN}{\mathcal{N}}

\newcommand{\I}{\mathcal{I}}

\newcommand{\EE}{\mathcal{E}}

\renewcommand{\L}{\mathcal{L}}
\renewcommand{\v}{{\mathbf v}}

\newcommand{\al}{\alpha}
\newcommand{\be}{\beta}
\newcommand{\dl}{\delta}
\newcommand{\Dl}{\Delta}
\newcommand{\eps}{\varepsilon}

\newcommand{\g}{\gamma}

\newcommand{\ld}{\lambda}
\newcommand{\Ld}{\Lambda}

\newcommand{\ft}{\widehat}
\newcommand{\wt}{\widetilde}
\newcommand{\cj}{\overline}
\newcommand{\dx}{\partial_x}
\newcommand{\dt}{\partial_t}
\newcommand{\dd}{\partial}

\newcommand{\0}{\mathbf 0}

\newcommand{\jb}[1]
{\langle #1 \rangle}

\renewcommand{\u}{\vec{{\mathbf u}}}
\renewcommand{\H}{\mathcal{H}}
\renewcommand{\rhoo}{\vec{{\mathbf \rho}}}

\newcommand{\les}{\lesssim}

\newcommand{\ind}{\mathbf 1}

\newcommand{\E}{\mathbb{E}}

\renewcommand{\o}{\omega}
\renewcommand{\O}{\Omega}

\numberwithin{equation}{section}
\numberwithin{theorem}{section}

\newtheorem*{ackno}{Acknowledgements}

\DeclareMathOperator{\Lip}{Lip}

\usepackage{tikz}

\usetikzlibrary{shapes.misc}
\usetikzlibrary{shapes.symbols}
\usetikzlibrary{shapes.geometric}
\tikzset{
	dot/.style={circle,fill=black,draw=black,inner sep=1pt,minimum size=0.5mm},
	>=stealth,
	}
\tikzset{
	ddot/.style={circle,fill=white,draw=black,inner sep=2pt,minimum size=0.8mm},
	>=stealth,
	}

\tikzset{decision/.style={ 
        draw,
        diamond,
        aspect=1.5
    }}

\tikzset{dia2/.style
={diamond,fill=white,draw=black,inner sep=0pt,minimum size=1mm},
	>=stealth,
	}

\tikzset{dia/.style
={star,fill=black,draw=black,inner sep=0pt,minimum size=1mm},
	>=stealth,
	}

\makeatletter
\def\DeclareSymbol#1#2#3{\expandafter\gdef\csname MH@symb@#1\endcsname{\tikz[baseline=#2,scale=0.15]{#3}}}
\def\<#1>{\csname MH@symb@#1\endcsname}
\makeatother

\DeclareSymbol{X}{-2.4}{\node[dot] {};}
\DeclareSymbol{1}{0}{\draw[white] (-.4,0) -- (.4,0); \draw (0,0)  -- (0,1.2) node[dot] {};}
\DeclareSymbol{2}{0}{\draw (-0.5,1.2) node[dot] {} -- (0,0) -- (0.5,1.2) node[dot] {};}

\DeclareSymbol{T0}{-2.7}
 { \draw (0,0) node[ddot]{};}

\DeclareSymbol{T1}{0}
 {  
 \draw (0,0) node[dot]{} -- (0,4) node[ddot] {}; 
 \draw (-3,0) node[dot] {} -- (0,4)node[ddot] {} -- (3,0) node[dot] {};
\draw[line width=1pt] (0, 4)node[ddot, label=above:$r_1$]{} 
.. controls(3,6) and (5,6) ..
(8, 4)node[ddot, label=above:$r_2$]{}
(4, 12) node[label ] {$\underline{j = 1}$};
 }

\DeclareSymbol{T2}{0}
 {\draw (0,0) node[dot]{} -- (0,4) node[ddot] {}; 
 \draw (-3,0) node[dot] {} -- (0,4)node[ddot] {} -- (3,0) node[dot] {};
 \draw (0,-4) node[dot]{} -- (0,0) node[dot] {}; 
 \draw (-3,-4) node[dot] {} -- (0,0)node[dot] {} -- (3,-4) node[dot] {};
\draw[line width=1pt] (0, 4)node[ddot, label=above:$r_1$]{} 
.. controls(3,6) and (5,6) ..
(8, 4)node[ddot, label=above:$r_2$]{}
(4, 12) node[label ] {$\underline{j = 2}$};
 }

\DeclareSymbol{T3}{0}
 {\draw (0,0) node[dot]{} -- (0,4) node[ddot] {}; 
 \draw (-3,0) node[dot] {} -- (0,4)node[ddot] {} -- (3,0) node[dot] {};
 \draw (0,-4) node[dot]{} -- (0,0) node[dot] {}; 
 \draw (-3,-4) node[dot] {} -- (0,0)node[dot] {} -- (3,-4) node[dot] {};
\draw (10,0) node[dot]{} -- (10,4) node[ddot] {}; 
 \draw (7,0) node[dot] {} -- (10,4)node[ddot] {} -- (13,0) node[dot] {};
\draw[line width=1pt] (0, 4)node[ddot, label=above:$r_1$]{} 
.. controls(4,6) and (6,6) ..
(10, 4)node[ddot, label=above:$r_2$]{}
(5, 12) node[label ] {$\underline{j = 3}$};
 }

\tikzstyle{dot1} = [ draw=  gray!00, 
 rectangle, rounded corners, fill=gray!00,  inner sep=1pt, inner ysep=3pt]

\tikzstyle{dot2} = [ draw=  black, 
ellipse, fill=gray!00,  inner sep=1pt, inner ysep=3pt]

\tikzstyle{dot3} = [ draw=  gray!00, 
ellipse, fill=gray!00,  inner sep=1pt, inner ysep=3pt]

\DeclareSymbol{T4}{0}
 {\draw (0,0) node[dot1]{$\phantom{n_2^{(1)} = n^{(2)}}$} -- (0,15) node[dot1] {$\phantom{n^{(1)}}$}; 
 \draw (-13,0) node[dot1] {$n_1^{(1)}$} -- (0,15)node[ddot] {$\phantom{n^{(1)}}$} 
 -- (13,0) node[dot1] {$n_3^{(1)}$};
 \draw (0,-15) node[dot1]{$n_2^{(2)}$} -- (0,0) node[dot1] {$\phantom{n_2^{(1)} = n^{(2)}}$}; 
 \draw (-13,-15) node[dot1] {$n_1^{(2)}$} -- (0,0)node[dot1] {$n_2^{(1)} = n^{(2)}$} -- (13,-15) node[dot1] {$n_3^{(2)}$};
\draw (40,0) node[dot1]{{$n_2^{(3)}$}} -- (40,15) node[dot3]  {$\phantom{n^{(1)} = n^{(3)}}$} ; 
 \draw (27,0) node[dot1] {$n_1^{(3)}$} -- (40,15)node[dot3]  {$\phantom{n^{(1)} = n^{(3)}}$}  -- (53,0) node[dot1] {$n_3^{(3)}$};
\draw[line width=1pt] (0, 15)node[ddot, label=above:$r_1$]{$n^{(1)}$} 
.. controls(10,23) and (27,23) ..
(40, 15)node[dot2, label=above:$r_2$] {$n^{(1)} = n^{(3)}$} ;
 }

\newsavebox{\sticka}
\savebox{\sticka}[4pt]{
\begin{tikzpicture}[x=1pt, y=1pt]
\draw[line width= 1] (0,0) -- (0,6); 
\path [draw, fill] (0,7.5) circle (1.5pt); 
\end{tikzpicture}
}

\def\stick{\mathchoice
    {\raisebox{-1pt}{\usebox{\sticka}}}%
    {\raisebox{-1pt}{\usebox{\sticka}}}%
    {}%
    {}}

\begin{document}

\baselineskip = 14pt

\title[Exponential ergodicity for SdSG]
{Exponential ergodicity for the stochastic hyperbolic sine-Gordon equation on the circle}

\author[K.~Seong]
{Kihoon Seong}

\address{Kihoon Seong\\
Department of Mathematics\\
Cornell University\\ 
310 Malott Hall\\ 
Cornell University\\
Ithaca\\ New York 14853\\ 
USA }

\email{kihoonseong@cornell.edu}

\subjclass[2020]{35L15, 37A25, 60H15.}

\keywords{Exponential ergodicity; Unique ergodicity; Gibbs measure; Stochastic hyperbolic sine-Gordon equation.}

\begin{abstract}
In this paper, we show that the Gibbs measure of the stochastic hyperbolic sine-Gordon equation on the circle is the unique invariant measure for the Markov process. Moreover, the Markov transition probabilities converge exponentially fast to the unique invariant measure in a type of 1-Wasserstein distance. The main difficulty comes from the fact that the hyperbolic dynamics does not satisfy the strong Feller property even if sufficiently many directions in a phase space are forced by the space-time white noise forcing. We instead establish that solutions give rise to a Markov process whose transition semigroup satisfies the asymptotic strong Feller property and convergence to equilibrium
in a type of Wasserstein distance.

\end{abstract}

\maketitle

\tableofcontents

\section{Introduction}
\label{SEC:1}

We consider the following stochastic damped nonlinear wave equation, so-called the sine-Gordon equation on $\T = \R/2\pi\Z$, with an additive space-time white noise forcing:
\begin{align}
\begin{cases}
\dt^2 u + \dt u + (1- \dx^2)  u   + \g \sin(\be u) = \sqrt{2}\xi\\
(u, \dt u) |_{t = 0}=\u_0=(u_0, v_0) , 
\end{cases}
\qquad (t, x) \in \R_+\times\T,
\label{SdSG}
\end{align}

\noi
where $\g$ and $\be$ are non-zero real numbers and $\xi$ denotes a (Gaussian) space-time white noise on $\R_+\times\T$ with the space-time covariance given by
\begin{align}
\E\big[ \xi(x_1, t_1) \xi(x_2, t_2) \big]
= \dl(x_1 - x_2) \dl (t_1 - t_2).
\label{white}
\end{align}

\noi
The main goal of this paper is to establish quantified ergodicity
for \eqref{SdSG} associated with the Gibbs measure, which can be formally written as
\begin{align}
d\rhoo(u,\dt u) = \frac 1Z e^{-H(u,\dt u) } du \otimes d(\dt u).
\label{Gibbs1}
\end{align}

\noi
Here,  $Z$ denotes the partition function\footnote{In the subsequent sections, we use $Z$ to represent different normalization constants.} and 
\begin{align}
H(u,\dt u)= \frac12\int_{\T}\big(|\dx u(x)|^2 + |\dt u(x)|^2+|u(x)|^2 \big)dx -\frac\g\be \int_{\T}\cos\big(\be u(x)\big)dx
\label{Hamil1}
\end{align}

\noi
denotes the Hamiltonian of the (deterministic undamped) sine-Gordon (SG) equation:
\begin{align}
\dt^2 u +  (1- \dx^2)  u   +  \g \sin(\be u) = 0,
\label{NLW1}
\end{align}

\noi
which preserves the Hamiltonian \eqref{Hamil1} and consequently the corresponding Gibbs measure at least formally. The sine nonlinearity present in \eqref{SdSG} exhibits a close connection to models originating from relativistic and quantum field theories \cite{PS, BEMS, Fro}, which has garnered significant attention in recent years. Regarding the deterministic sine-Gordon equation \eqref{NLW1}, McKean \cite{McKean, Mc94} focused on the one-dimensional case and established an invariant Gibbs measure for \eqref{NLW1} on $\T$. The next natural question is the uniqueness of the invariant measure under the flow of \eqref{NLW1} (i.e.~unique ergodicity), which implies\footnote{As indicated by Birkhoff's ergodic theorem} that the solution of \eqref{NLW1} converges to the Gibbs measure $\rhoo$ in the sense that for every $\u_0\in \supp(\rhoo)$\footnote{$\supp(\rhoo) \subset \H^{\frac 12-}(\T)$,  where $\H^{\frac 12-}(\T)$ is the Sobolev space defined in \eqref{SOBOLEV}. } and function $f$ bounded and measurable,
\begin{align}
\lim_{T\to \infty} \frac 1T\int_0^T f(\Phi^{\text{SG}}_t(\u_0) ) dt =\int f(\u) d\rhoo(\u)
\label{BIRK}
\end{align}

\noi
where $\Phi_t^{\text{SG}}$ is the flow map of \eqref{NLW1}. We, however, point out that  the Gibbs measure $\rhoo$ is not uniquely ergodic under \eqref{NLW1}. See Remark \ref{REM:nonunique} for the explanation. Regarding the study of deterministic sine-Gordon equation, see also \cite{Dick, BFLT, BL86, GT87, KW92, WZ97, CKR04}.

The main result of the present article is that as time approaches infinity, solutions of stochastic PDE \eqref{SdSG} converge to the unique invariant measure with exponential rate. The reason for the significance of the stochastic equation \eqref{SdSG} lies in its connection to the PDE construction of the Euclidean quantum field theory measure, so-called stochastic quantization \cite{PW, RSS}. It is evident that \eqref{SdSG} can be expressed as a combination of the deterministic nonlinear wave dynamics \eqref{NLW1} and the Ornstein-Uhlenbeck dynamics for momentum $\dt u$:
\begin{align}
\dt (\dt u) = - \dt u + \sqrt 2 \xi,
\label{Lan}
\end{align}

\noi
which preserves the spatial white noise measure
\begin{align}
d\mu_0(v)=\frac 1Ze^{-\frac 12 \int v^2 dx} dv.
\label{white00}
\end{align}

\noi
Therefore, the Gibbs measure $\rhoo$ \eqref{Gibbs1}  
\begin{align}
d\rhoo(u, \dt u )&= \frac 1Z \exp\bigg( \frac \g\be \int_\T \cos\big(\be u \big) dx-\int_\T(\dx u)^2 dx-\int_\T u^2 dx -\frac 12 \int_{\T} (\dt u)^2dx \bigg) du \otimes d(\dt u) \notag \\
&=\frac 1Z \exp\bigg( \frac \g\be \int_\T \cos\big(\be u \big) dx \bigg) d\muu(u,\dt u)
\label{GibbsW}
\end{align}

\noi 
is expected to be invariant under \eqref{SdSG}, where $\muu=\mu\otimes \mu_0$. Here, $\mu$ is the massive Gaussian free field
\begin{align}
d\mu(u)=\frac 1Ze^{-\frac 12 \int (\dx u)^2+ u^2 dx} du
\label{mass0}
\end{align}

\noi
and $\mu_0$ is the white noise measure in \eqref{white00} (see Subsection \ref{SUBSEC:GFF} for each definition). The Gibbs measure \eqref{GibbsW} corresponds to the well-known model of quantum field theory, so-called the sine-Gordon measure, which can be defined by not performing renormalization, specifically for $d=1$. 
Considering the stochastic quantization perspective, the dynamical model \eqref{SdSG} corresponds to the ``canonical" stochastic quantization \cite{RSS} of the quantum sine-Gordon model, which is characterized by the Gibbs measure $\rhoo$ in \eqref{GibbsW}. 

As we mentioned above, the main result in this paper is the exponential convergence of the Markov transition probabilities to the (unique) invariant Gibbs measure $\rhoo$ for the system \eqref{SdSG}. We note that (unique) ergodicity for stochastic nonlinear wave equations has been studied extensively in \cite{BD, BDT, BOS2016, GHMR, CEGLA,TOL, FT, Ngu}. Specifically,  in \cite{GHMR}, Glatt-Holtz, Mattingly, and Richards studied ergodic invariant measures (but not Gibbs type measures) and its uniqueness for the stochastic sine-Gordon equation with smoother noise. We want to highlight that our noise $\xi$ in \eqref{SdSG} is very rough\footnote{Its spatial regularity is $C^{-\frac 12-\eps}(\T)$.}, meaning that it is just a distribution-valued random process. We, however, notice that this rough noise results in the Gibbs measure $\rhoo$ being invariant under the flow of \eqref{SdSG}. We now state our main result as follows.

\begin{theorem}\label{THM:1}
The Gibbs measure $\rhoo$ in \eqref{GibbsW} is the unique invariant measure for 
the Markov semigroup $\{\Pt{t}\}_{t \ge 0}$ associated with the flow of \eqref{SdSG}. Moreover,  there exist $T>0$ and $c>0$ such that\footnote{Here, $\Pt{t}^*$ denote the adjoint Markov semigroup $\Pt{t}^*$ on the space of probability measures.
Then, $\Pt{t}^*\nu$ can be interpreted as the probability distribution of the solution process, $\nu$ being the initial distribution.}
\begin{align}
\W_{ d}\big( \Pt{t}^*\nu, \rhoo \big)\le e^{-c \lfloor \frac t{T} \rfloor } \W_{ d}(\nu, \rhoo)
\label{EXPWASS}
\end{align}

\noi
for every $t\ge T$ and any $\nu\in \mathcal{P}r(\H^{\frac 12-})$\footnote{the space of probability measures on $\H^{\frac 12-}(\T)$, where $\H^{\frac 12-}(\T)$ is the Sobolev space defined in \eqref{SOBOLEV}. }, where $\W_{ d}$ is a type of 1-Wasserstein distance defined in \eqref{STARCONTRAC}.

\end{theorem}

As explained before, the hyperbolic equation \eqref{SdSG} is composed of the deterministic Hamiltonian PDE \eqref{NLW1} and the Langevin dynamics \eqref{Lan}. Therefore, Bourgain’s invariant measure argument \cite{BO94} in the stochastic setting gives the invariance of the measure under the flow \eqref{SdSG} (see Section \ref{SEC:Invar}). In the case of stochastic PDEs with (nondegenerate) noise\footnote{i.e.~sufficiently many directions in  phase space are forced by the noise.}, the study of the strong Feller property is tightly related to proving unique ergodicity and convergence to the unique invariant measure in total variation. For example, in parabolic stochastic PDEs, the strong Feller property is to be naturally expected when an additive noise forces every direction in Fourier space like space-time white noise. We, however, point out that the hyperbolic dynamics \eqref{SdSG} do not satisfy the strong Feller property (see Proposition \ref{PROP:fail} below) even though  the space-time white noise forces  all of Fourier modes. 
Due to the failure of strong Feller property, the method of \cite{DZ, TW, HM18} cannot be exploited here. To address this issue, we employ the approach introduced by Hairer and Mattingly \cite{HM06}, which involves the concept of asymptotic strong Feller to handle Markov semigroups that do not satisfy the strong Feller property. The asymptotic strong Feller property established in proving the unique ergodicity (in Theorem \ref{THM:1})  prescribes some kind of smoothing property at time $\infty$ in proper Wasserstein-1 distances approximating the total variation distance. This is a big difference compared to the strong Feller property which exhibits a smoothing property at a fixed time $t>0$ in the total variation distance. In our setting, because of the failure of the strong Feller property, the strong convergence in the total variation distance is not expected to hold. Therefore, instead of proving the existence of a spectral gap in the total variation distance, we show that the Markov semigroup $\Pt{t}$ has a contraction property in a particular 1-Wasserstein distance (see \eqref{STARCONTRAC} below), which implies the exponential convergence to the Gibbs measure $\rhoo$ for \eqref{SdSG}. In particular, by taking $\nu=\dl_{\u_0}$ in \eqref{EXPWASS}, one can see that the Markov transition probabilities $\Pt{t}^*\dl_{\u_0}$ converge exponentially fast to the Gibbs measure.

Based on the strategy introduced in \cite{HMS}, the geometric ergodicity follows from  suitable large-time smoothing estimates (see Lemma \ref{LEM:gra} and \ref{LEM:gra1} below) for the Markov process, related to asymptotic strong Feller property, together with a suitable irreducibility condition (see Lemma \ref{LEM:visit} below) for system \eqref{SdSG}. More precisely, with large-time smoothing estimates, we (i) establish a contracting distance  for the Markov semigroup $\{\Pt{t} \}_{t\ge 0}$ when initial states are not far from each other (see Proposition \ref{PROP:contract} below) and (ii) show that bounded sets in the support of the Gibbs measure are small in the sense of Proposition \ref{PROP:dsmall}. Once (i) and (ii) are established, exponential convergence towards the Gibbs measure \eqref{GibbsW} is obtained through a Wasserstein  distance (see \eqref{STARCONTRAC} below). The main ingredient to prove part (i) is the gradient estimate (Lemma \ref{LEM:gra}) whose proof is based on the Malliavin calculus. We point out that another important ingredient to quantify the convergence rate is to prove the existence of a Lyapunov function for \eqref{SdSG}. Note that, unlike parabolic equations, nonlinear wave equations do not have a strong dissipation and so the existence of a Lyapunov function is not a priori clear. As in extensive literature (say, \cite{CKSTT00, CKSTT2003}) about dispersive/hyperbolic equations, we carry out our energy estimates with a modified energy technique. The introduction of a modified energy term is an essential ingredient to exhibit a hidden smoothing effect for dispersive/hyperbolic equations.
Notice that, as we emphasized above, the irreducibility condition and Lyapunov structure play a crucial role in obtaining the quantified ergodicity result \eqref{EXPWASS}. In Appendix \ref{SEC:APP}, based on the method in \cite{FT}, by only exploiting the large-time smoothing estimates (see Proposition \ref{PROP:diff} below), it can be shown that the Gibbs measure is uniquely ergodic for \eqref{SdSG} without quantifying the convergence rate.

\begin{remark}\rm\label{REM:nonunique}
We point out that the unique ergodicity for the deterministic sine-Gordon equation \eqref{NLW1} does not hold. As we already explained above, the sine-Gordon equation is an integrable Hamiltonian system which has an infinite number of conserved quantities. Therefore,  at least formally, there are infinitely many invariant measures.
For example, important conserved quantities are the total momentum $M(u,\dt u)=\int \dt u \dx u dx$  as well as the Hamiltonian \eqref{Hamil1}. Hence, $e^{-M(u,\dt u)}du \otimes d(\dt u)$ is also invariant under the flow of \eqref{NLW1}.

\end{remark}

\begin{remark}\rm
Note that the damping term $\dt u$ in \eqref{SdSG} plays an essential role when considering the invariant measure for \eqref{SdSG}. Notice that the linear dynamics  
\begin{align}
\dt^2 u + \dt u + (1- \dx^2)  u  = \sqrt{2}\xi
\label{lin00}
\end{align}

\noi
can be written as a superposition
of the deterministic linear wave dynamics \eqref{NLW1} 
which preserves the massive Gaussian free field $\mu$ defined in \eqref{mass0}, and the Ornstein-Uhlenbeck dynamics of $\dt u$, namely \eqref{Lan}
which preserves the spatial white noise measure $\mu_0$ defined in \eqref{white00}.
Hence, the distribution of $(u(t), \dt u(t))$ is invariant under the linear dynamics \eqref{lin00}. Notice that when considering the problem without the damping term $\dt u$, there is no invariant measure for the linear dynamics due to the non-conservative nature of the problem. In particular, the variance of the solution $u(t)$ is time dependent in the absence of the damping term $\dt u$, unlike the case of \eqref{SdSG} which contains $\dt u$. 
See Subsection \ref{SUBSEC:Stoconvol} for the role of $\dt u$ in proving the invariance of the Gibbs measure.

\end{remark}

\begin{remark}\rm
In order to avoid a problem\footnote{Since the zeroth frequency is not
controlled due to the lack of the conservation of the $L^2$-mass under the dynamics.} at the zeroth frequency under the Gibbs measure construction, we need to work with the massive linear part $\dt^2u+(1-\dd_x^2)u$. Hence, because of this reason, we work with the massive case in this paper.
\end{remark}

\begin{remark}\rm
To the best of the author's knowledge, the exponential convergence towards the Gibbs measure for ``canonical'' stochastic quantization equations has not been addressed even if the qualitative result on ergodicity was known in \cite{TOL, TOLPRE} where the unique ergodicity of $\Phi^4_d$-measure was proved under the hyperbolic $\Phi^4_d$-model when $d=1,2$.
\end{remark}

\begin{remark}\rm
In \cite{ORSW, ORSW2}, the hyperbolic stochastic sine-Gordon model on $\T^2$ was studied.
In particular, in \cite{ORSW2} they established almost sure global well-posedness of \eqref{SdSGT2} and invariance of the (renormalized) Gibbs measure for\footnote{
The behavior of the model \eqref{SdSGT2} is also influenced by the value of $\be^2 > 0$, similar to the parabolic case. See \cite{HaS, CHaS}.} $0\le \be^2<2\pi$. Regarding the construction of the (sine-Gordon) Gibbs measure on the infinite volume $\R^2$, see also \cite{BAR}. Similar to the stochastic nonlinear wave equation with a polynomial nonlinearity, it is necessary to perform a suitable renormalization in order to give meaning to the equation on $\T^2$ as follows:
\begin{equation}
\dt^2u+ \dt u + (1-\Dl)u + \infty\cdot \sin(\be u) = \sqrt{2}\xi.
\label{SdSGT2}
\end{equation}

\noi
The requirement for renormalization becomes apparent when we consider the rapid oscillations of $\sin (\be u)$ caused by the expected roughness of $\xi$ and so $u$ on $\T^2$, leading the nonlinearity $\sin (\be u)$ to tend towards zero in some limiting sense. In order to compensate such decay and have a non-trivial solution, we should take $\g$ in \eqref{SdSG} as $ \g=\infty$. See \cite[Proposition 1.4]{ORSW} for a triviality phenomenon for the unrenormalized model. When trying to extend the quantified ergodicity result in Theorem \ref{THM:1} to \eqref{SdSGT2}, the primary challenge arises from the non-polynomial nature of the nonlinearity, rendering the analysis of the pertinent stochastic objects significantly more intricate in comparison to the one-dimensional case.
We plan to study the extension of Theorem \ref{THM:1} to the case $d=2$ in a future work.

\end{remark}

\subsection{Organization of the paper} 
In Section \ref{SEC:Invar}, we prove the invariance of the Gibbs measure under the flow of \eqref{SdSG}.
In Section \ref{SEC:exp}, we present the failure of strong Feller property (Proposition \ref{PROP:fail}) and show the exponential ergodicity in a type of 1-Wasserstein distance. 
In Appendix \ref{SEC:APP}, we present that the Gibbs measure is uniquely ergodic for \eqref{SdSG} by only exploiting the large-time smoothing estimates. 

\section{Invariance of the Gibbs measure under the flow}
\label{SEC:Invar}

\subsection{Truncated Gibbs measures and its convergence}
\label{SUBSEC:GFF}
In this subsection, we consider the Gibbs measure $\rhoo$. To achieve the main purpose, we begin by establishing specific notations. Consider $ \al \in \R$, and let $\mu_\al$ represent a Gaussian measure with the Cameron-Martin space $H^\al(\T)$, formally defined as follows:
\begin{align}
d\mu_\al= Z_\al^{-1} e^{-\frac 12 \| u\|_{{H}^{\al}}^2} du
& =  Z_\al^{-1} \prod_{n \in \Z} 
 e^{-\frac 12 \jb{n}^{2\al} |\ft u(n)|^2} d\ft u(n) , 
\label{gauss0}
\end{align}

\noi
where $\jb{\,\cdot\,} = (1+|\,\cdot\,|^2)^\frac{1}{2}$. For $\al = 1$, the Gaussian measure $\mu_1$ corresponds to the massive Gaussian free field, whereas it corresponds to the white noise measure $\mu_0$ for $\al = 0$. In particular, the massive Gaussian free field $\mu$ and white noise $\mu_0$ are Gaussian measures on $\mathcal S'(\T)$ with the following covariances
\begin{align*}
&\int \jb{f,\phi} \jb{g,\phi} d\mu(\phi)=\jb{f,(1-\Dl)^{-1}g}_{L^2(\T)},\\
&\int \jb{f,\phi} \jb{g,\phi} d\mu_0(\phi)= \langle f, g \rangle_{L^2 (\T)}
\end{align*}

\noi
for any $f,g \in \mathcal{S}(\T)$. To keep things straightforward, we set 
\begin{align}
\mu = \mu_1
\qquad \text{and}
\qquad 
\muu = \mu \otimes \mu_{0} .
\label{gauss1}
\end{align}

\noi
Define the index sets $\Ld$ and $\Ld_0$ by 
\begin{align}
\Ld = \Z_+ \qquad \text{and}\qquad \Ld_0 = \Ld \cup\{0\}
\label{index}
\end{align}

\noi
such that $\Z = \Ld \cup (-\Ld) \cup \{0\}$.
Then, let $\{ g_n \}_{n \in \Ld_0}$ and $\{ h_n \}_{n \in \Ld_0}$ be sequences of mutually independent standard complex-valued\footnote {This means that $g_0,h_0\sim\NN_\R(0,1)$
and $\Re g_n, \Im g_n, \Re h_n, \Im h_n \sim \NN_\R(0,\tfrac12)$ for $n \ne 0$.} Gaussian random variables and set $g_{-n} := \cj{g_n}$ and $h_{-n} := \cj{h_n}$ for $n \in \Ld_0$.
Furthermore, we assume that $\{ g_n \}_{n \in \Ld_0}$ and $\{ h_n \}_{n \in \Ld_0}$ are independent from the space-time white noise $\xi$ in \eqref{white}. We now define random fields $u= u^\o$ and $v = v^\o$ by the following  Gaussian Fourier series:\footnote{As a convention, we equip $\T$ with the normalized Lebesgue measure denoted as $dx_{\T} = (2\pi)^{-1} dx$.}
\begin{equation} 
u^\o(x) = \sum_{n \in \Z} \frac{ g_n(\o)}{\jb{n}} e^{inx} \qquad \text{and} \qquad
v^\o(x) = \sum_{n\in \Z}  h_n(\o) e^{inx}. 
\label{ranseries}
\end{equation}

\noi
Let $\Law(X)$ denote the probability distribution of a random variable $X$ under the underlying probability measure $\PP$. Consequently, we obtain the following:
\begin{align}
\Law (u, v) = \muu = \mu \otimes \mu_0
\end{align}

\noi
for $(u, v)$ in \eqref{ranseries}.
Note that  $\Law (u, v) = \muu$ is supported on
\begin{align}
\H^{\al}(\T): = H^{\al}(\T)\times H^{\al- 1}(\T)
\label{SOBOLEV}
\end{align}

\noi
for $\al < \frac 12$ but not for $\al \geq \frac 12$
(and more generally in $W^{\al, p}(\T) \times W^{\al-1, p}(\T)$
for any $1 \le p \le \infty$ and $\al < \frac 12$). We represent the notation $a-$ as $a-\eps$, where $\eps$ can be made arbitrarily small. Hence, we can write $\supp(\rhoo)\subset \H^{\frac 12-}(\T)$. 

Consider $\P_N$ as a smooth frequency projector onto the frequencies $\{n\in\Z :|n|\leq N\}$, defined as a Fourier multiplier operator in the following manner:
\begin{align}
\varphi_N(n) = \varphi(N^{-1}n)
\label{chi}
\end{align}

\noi
for some  fixed non-negative function  $\varphi \in C^\infty_c(\R)$ 
such that $\supp \varphi \subset \{\xi\in\R:|\xi|\leq 1\}$ and $\varphi\equiv 1$ 
on $\{\xi\in\R:|\xi|\leq \tfrac12\}$. Given $N\in \N$, we define the truncated Gibbs measure 
\begin{align}
d\rhoo_N(u,\dt u)= Z_N^{-1} e^{\frac {\g}{\be}\int_\T \cos(\be \P_N u) dx }   d\muu(u, \dt u)
\label{TruGibbs}
\end{align}

\noi
where $Z_N$ denotes the partition function. In the following subsections, we show that the truncated Gibbs measure \eqref{TruGibbs} is an invariant measure for the truncated dynamics \eqref{SdSG2} and so the invariance of the limiting Gibbs measures $\rhoo$ under the flow \eqref{SdSG} comes as a corollary.
Before moving to the next subsection, we consider the following proposition.
\begin{proposition}\label{PROP:Gibbs}
Let $\be,\g \in \R\setminus\{0\}$. Given any finite $ p \ge 1$, 
there exists $C_p > 0$ such that 
\begin{equation}
\sup_{N\in \N} Z_N=\sup_{N\in \N}\int e^{\frac {\g p}{\be}\int_\T \cos(\be\P_N u) dx }  d\muu(u, \dt u)
\leq C_p  < \infty.
\label{exp1}
\end{equation}

\noi
Moreover, we have
\begin{equation}
\lim_{N\rightarrow\infty}e^{\frac {\g}{\be }\int_\T \cos(\be \P_N  ) dx }=e^{\frac {\g}{\be}\int_\T \cos( \be \u ) dx } \qquad \text{in } L^p(\muu).
\label{exp2}
\end{equation}

\noi
As a consequence, the truncated Gibbs measure $\rhoo_N$ in \eqref{TruGibbs} converges\footnote{This convergence allows us to define the Gibbs measure $\rhoo$ as the limit in total variation of $\rhoo_N$.}, in the sense of \eqref{exp2}, to the Gibbs measure $\rhoo$ given by
\begin{align}
d\rhoo(u, \dt u)= Z^{-1} e^{\frac {\g}{\be}\int_\T \cos(\be u) dx }   d\muu(u, \dt u).
\label{Gibbslim}
\end{align}

\noi
Furthermore, 
the resulting Gibbs measure $\rhoo$ is equivalent 
to the base Gaussian field $\muu$.
\end{proposition}

Proposition \ref{PROP:Gibbs} allows us to define the Gibbs measure 
\begin{align*}
d\rhoo(u,v)= Z^{-1} e^{\frac {\g}{\be}\int_\T \cos(\be u) dx }   d\muu(u, v)
\end{align*}

\noi
as a limit of the truncated Gibbs measure 
\begin{align*}
d\rhoo_N(u,v)= Z_N^{-1} e^{\frac {\g}{\be}\int_\T \cos(\be \P_N u) dx }   d\muu(u, v).
\end{align*}

\noi
Once the uniform exponential integrability \eqref{exp1} is shown, which can be easily checked, then the desired convergence \eqref{exp2} of the density follows from a standard argument. See \cite[Remark 3.8]{TZVET}.

\subsection{Stochastic convolution and its invariant measure}
\label{SUBSEC:Stoconvol}

Denote $\Psi$ as the stochastic convolution, which satisfies the linear stochastic damped wave equation as follows:
\begin{align}
\begin{cases}
\dt^2 \Psi + \dt\Psi +(1-\dx^2)\Psi  = \sqrt{2}\xi\\
(\Psi,\dt\Psi)|_{t=0}=\u_0=(u_0,v_0),
\label{SDLW}
\end{cases}
\end{align}

\noi
where  $(u_0, v_0) = (u_0^\o, v_0^\o)$ 
is a pair of  the Gaussian random distributions with 
$\Law (u_0^\o, v_0^\o) = \muu = \mu \otimes \mu_0$ in \eqref{gauss1}.
We set the linear damped wave propagator $\D(t)$ by 
\begin{equation} 
\D(t) = e^{-\frac t2 }\frac{\sin\Big(t\sqrt{\tfrac34-\dx^2}\Big)}{\sqrt{\tfrac34-\dx^2}}
\label{D1}
\end{equation}

\noi
viewed as a Fourier multiplier operator. Then, the stochastic convolution $\Psi$ can be expressed as 
\begin{equation}
\Psi (t) =  S(t)(u_0, v_0) + \sqrt 2 \int_0^t\D(t-t')dW(t'), 
\label{W2}
\end{equation}

\noi
where  $S(t)$ is defined by 
\begin{align}
S(t) (f, g) = \dt\D(t)f +  \D(t) (f + g)
\label{St0}
\end{align}

\noi
and 
$W$ represents  a cylindrical Wiener process on $L^2(\T)$: 
\begin{align}
W(t) =  \sum_{n \in \Z } B_n (t) e^{inx},
\label{CW}
\end{align}

\noi
and  $\{ B_n \}_{n \in \Z}$ is defined by $B_n(0) = 0$ and $B_n(t) = \jb{\xi, \ind_{[0, t]} \cdot e_n}_{ t, x}$. In this context, the notation $\jb{\cdot, \cdot}_{t, x}$ represents the duality pairing on $\R \times \T$ and $e_n(z):=e^{in\cdot z}$. Consequently, the family $\{ B_n \}_{n \in \Z}$ consists of mutually independent complex-valued Brownian motions\footnote {In particular, $B_0$ is a standard real-valued Brownian motion.} conditioned such that $B_{-n} = \cj{B_n}$ for all $n \in \Z$. Specifically, we have normalized $B_n$ such that $\text{Var}(B_n(t)) = t$. It is evident that $\Psi$ lies in $C (\R_+;W^{\frac{1}{2}-\eps, \infty}(\T))$ almost surely for any $\eps > 0$. Below, we take the symbol $\eps > 0$ to represent a sufficiently small positive constant, capable of being close to zero. 

In the subsequent discussions, we will adhere to Hairer's convention, representing stochastic terms using trees. The vertex denoted as ``\,$\<X>$\,''  corresponds to the space-time white noise $\xi$, while the edge represents the Duhamel integral operator $\I$
given by 
\begin{align}
\I (F) (t) = \int_0^t \D(t - t') F(t') dt'
= \int_0^t e^{-\frac {t-t'}2 }\frac{\sin\Big((t-t')\sqrt{\tfrac34-\dx^2}\Big)}{\sqrt{\tfrac34-\dx^2}}
F(t') dt'.
\label{lin1}
\end{align}

\noi
We set  
\begin{align}
\Psi(t)=S(t)\u_0+\stick_{t}(\xi),
\label{W2a}
\end{align}

\noi
where $\Psi$ is as in \eqref{W2}, with the understanding that $\stick_t(\xi)$ in \eqref{W2a} denotes the second term in \eqref{W2}. As mentioned above, $\stick_t(\xi)$ has (spatial) regularity\footnote{Our focus is only on examining the spatial regularities of stochastic objects.} $\frac 12 -$. For future use, we note that for any $\al \in \R$ the linear propagator $\vec S(t)$ satisfies
\begin{align}
\|\vec S(t)\u_0 \|_{\H^\al} \les e^{-\frac t2}\|\u_0 \|_{\H^\al}.
\label{linest}
\end{align}

\noi
where $\vec S(t)\u_0=(S(t) \u_0, \dt S(t) \u_0)$ and $\u_0=(u_0,v_0)$. 

We first prove the invariance of $(\P_N)_\# \muu$ for the solution $\P_N\Psi$ to truncated linear stochastic damped wave equations 
\begin{align}
d\begin{pmatrix}
u_N\\v_N
\end{pmatrix} = -\begin{pmatrix}
0& -1\\ 1-\dx^2&0
\end{pmatrix}\begin{pmatrix}
u_N\\v_N
\end{pmatrix}dt + \begin{pmatrix}
0\\- v_Ndt+\sqrt{2}\P_NdW
\end{pmatrix}
\label{TSDLW}
\end{align} 

\noi
with data given by $(u_N,v_N)\big|_{t=0}=\P_N(u_0,u_1)$ whose law is $(\P_N)_{\#}\muu$, where $u_N=\P_N u$ and $v_N=\P_N v$. 

\begin{proposition}\label{PROP:InvarGauss}
The truncated Gaussian field $(\P_N)_\# \muu$ 
\begin{align*}
d(\P_N)_\#\muu=Z^{-1}e^{-\frac 12 \int (\dx u_N)^2+ u_N^2 dx}du_N  \otimes e^{-\frac 12 \int v_N^2 dx}dv_N
\end{align*}

\noi
is invariant under the flow \eqref{TSDLW}.

\end{proposition}


\begin{proof}
Establishing the invariance of $\muu_N=(\P_N)_{\#}\muu$ is the same as demonstrating that $\L_N^\#\muu_N=0$, where $\L_N$ represents the infinitesimal generator of \eqref{TSDLW}, and $\L_N^\#$ denotes its dual operation on probability measures on $E_N\times E_N$ as follows:
\begin{equation*}
\int_{E_N\times E_N}F(u,v)d(\L_N^\#\muu_N) = \int_{E_N\times E_N}(\L_NF)(u,v)d\muu_N(u,v)
\end{equation*}

\noi
for every $F\in C^{\infty}_b(E_N\times E_N)$, where $E_N=\text{span}\{e^{inx}: |n| \le N \}$. Considering \eqref{TSDLW}, we can express $\L_N$ as the sum of $\L_N^1$ and $\L_N^2$  (i.e.~$\L_N = \L_N^1+\L_N^2$), where $\L_N^1$ corresponds to the generator for linear wave equations, and $\L_N^2$ represents the generator for an Ornstein-Uhlenbeck process. Specifically, \eqref{TSDLW} can be viewed as a system of  stochastic differential equations in terms of the coordinates\footnote{For simplicity, we assume that $a_n$ and $b_n$ are real numbers. In general, the equation can be written with $\Re a_n, \Re b_n, \Im a_n$ and $\Im b_n$.  } $a_n=\ft u(n)$ and $b_n=\ft v(n)$, given by
\begin{align}
\begin{cases}
d a_n = b_n dt\\
d b_n = -\jb{n}^2a_ndt +(-\ b_ndt+\sqrt{2}dB_n)
\end{cases},|n|\le N
\label{SPUS0}
\end{align}

\noi
The infinitesimal generator for \eqref{SPUS0} is given by\footnote{See \cite[Section 3]{ORTZVET}}
\begin{align*}
\L_Nf(a_0,...,a_{N},b_0,...,b_{N}) = \sum_{n=0}^{N}b_n\partial_{a_n}f-\jb{n}^2a_n\partial_{b_n}f - b_n\partial_{b_n}f+\partial_{b_n}^2f.
\end{align*}

\noi
We now denote
\begin{align*}
\L_N^2f := \sum_{n=0}^{N}-b_n\partial_{b_n}f+\partial_{b_n}^2f.
\end{align*}

\noi
Subsequently, we can identify $\L_N^2$ as the generator of the following Ornstein-Uhlenbeck process
\begin{align*}
b_n(t)=e^{-t}b_n(0)+\sqrt{2}\int_0^te^{-(t-t')}dB_n(t'), \quad |n|\le N
\end{align*}

\noi
and a simple computation exploiting It\^o's isometry shows that $b_n$ is a centered (i.e.~mean $0$) Gaussian random variable with variance
\begin{align*}
\E(b_n(t)^2)=e^{-2t}\E(b_n(0)^2)+1-e^{-2 t}=1.
\end{align*}

\noi
In terms of \eqref{ranseries}, $\E(b_n(t)^2)= \E(b_n(0)^2) $=1 for every $t \ge 0$, which means that $\L_N^2$ preserves $(\P_N)_{\#}\mu_0$ (and so $(\P_N)_{\#}\muu$\footnote{The measure $(\P_N)_{\#}\muu$ decouples into the massive Gaussian free field $(\P_N)_{\#}\mu$ on $u$ and the white noise measure $(\P_N)_{\#}\mu_0$  on $v$. In particular, $\L_N^2$ only acts on the component $v$.}). On the contrary, we find that the generator of the truncated linear wave equations represented in terms of the Hamiltonian system of ODEs 
\begin{align*}
\begin{cases}
\frac{d}{dt}a_n = b_n,\\
\frac{d}{dt}b_n = -\jb{n}^2a_n,
\end{cases}~,|n|\le N
\end{align*}

\noi
is given by
\begin{align*}
\L_N^1 = \sum_{n=0}^{N}b_n\partial_{a_n}-\jb{n}^2a_n\partial_{b_n}.
\end{align*}

\noi
The Hamiltonian, which represents the energy  
\begin{align*}
H_N(a_0,...,a_{N},b_0,...,b_{N}) &=\frac12\sum_{n=0}^{N}\big(\jb{n}^2|a_n|^2+|b_n|^2\big)
\end{align*}

\noi
remains conserved under the linear flow, and thanks to Liouville's theorem, this flow preserves the Lebesgue measure $\prod_{n=0}^{N}da_ndb_n$. Thus, we observe that the measure  
\begin{align*}
e^{-\frac 12 \int (\dx u_N)^2+ u_N^2 dx}du_N \otimes e^{-\frac 12 \int v_N^2 dx}dv_N =e^{-H_{N}(a_0,...,a_{N},b_0,...,b_{N})}\prod_{n=0}^{N}da_ndb_n
\end{align*}

\noi
is also preserved under the Hamiltonian flow, which implies $(\L_N^1)^\# \muu_N=0$. By combining all results, we obtain  $\L_N^\#\muu_N=(\L_N^1+\L_N^2)^\#\muu_N=0$ which concludes the proof of the invariance under the truncated flow \eqref{TSDLW}.

\end{proof}

\begin{corollary}\label{COR:InvarGauss}
The Gaussian field $\muu = \mu \otimes \mu_0$ is invariant under the flow of linear stochastic wave equation \eqref{SDLW}.
\end{corollary}

\begin{proof}

By the almost sure convergence of $(\P_N\Psi,\P_N\dt\Psi)(t)$ towards $(\Psi,\dt\Psi)(t)$ in $\H^{\frac 12-}(\T)$ for any fixed $t \ge0$, we have the following weak convergence 
\begin{align}
\Law_{\PP}\big((\P_N\Psi,\P_N\dt\Psi)(t)\big) \to \Law_{\PP}\big( (\Psi,\dt\Psi)(t)\big). 
\label{PUS1}
\end{align}

\noi
Thanks to the invariance of $(\P_N)_\#\muu$ for $(\P_N\Psi,\P_N\dt\Psi)$ (Proposition \ref{PROP:InvarGauss}), we have 
\begin{align}
(\P_N)_\#\muu=\Law_{\PP}\big( (\P_N \Psi,\dt \P_N\Psi)(0)\big)=\Law_{\PP}\big( (\P_N\Psi,\dt \P_N\Psi)(t)\big)
\label{PUS2}
\end{align}

\noi
for every $t \ge 0$.  By combining \eqref{PUS1}, \eqref{PUS2}, and weak convergence of $(\P_N)_\#\muu$ towards $\muu$ which follows from the convergence  of the series  \eqref{ranseries} in $L^p(\O;\H^{\frac 12-}(\T))$ for any $p\ge 1$, we obtain
\begin{align*}
\muu=\Law\big( (\Psi,\dt\Psi)(t)\big)
\end{align*}

\noi
for every $t\ge 0$. This completes the proof.
\end{proof}

\subsection{On the truncated dynamics}
In this subsection, we prove the invariance of the Gibbs measure $\rhoo$ \eqref{Gibbslim} under the flow of \eqref{SdSG}. We first introduce the following truncated dynamics
\begin{align}
\begin{cases}
\dt^2 u_N + \dt u_N + (1- \dx^2)  u_N   + \g\P_N\big\{ \sin(\be\P_N u_N) \big\} = \sqrt{2}\xi\\
(u, \dt u) |_{t = 0} =\u_0= (u_0, v_0).
\end{cases}
\label{SdSG2}
\end{align}

\noi
When $N=\infty$, it is understood as $\P_N=\Id$ and so \eqref{SdSG2} becomes \eqref{SdSG}. In Proposition \ref{PROP:Finvar} we show that the truncated Gibbs measure 
\begin{align}
d\rhoo_N(\u_0)=Z_{N}^{-1}e^{\frac{\g}{\be} \int_\T \cos(\be \P_N u) dx } d\muu(\u_0)
\label{TGibbs}
\end{align} 

\noi
is a corresponding invariant measure under the flow of \eqref{SdSG2}.
Then, the invariance of the measure $\rhoo$ comes as a corollary (see Corollary \ref{COR:InvarGauss}) by combining the result for the truncated dynamics and the convergence of $\rhoo_N$ to $\rhoo$.

Consider $u_N$ as the solution to \eqref{SdSG2}. Then, by writing $u_N$ as 
\begin{align}
u_N = v_N + \Psi
= (v_N +  \P_N\Psi) + \P_N^\perp \Psi,
\label{decomp2}
\end{align}

\noi
where $\P_N^\perp = \Id - \P_N$, we see that the dynamics of the truncated SdNLW \eqref{SdSG2} decouples into the linear dynamics for the high frequency part given by $\P_N^\perp \Psi$ 
\begin{align}
\dt^2 \P_N^\perp\Psi + \dt  \P_N^\perp\Psi +(1-\dx^2) \P_N^\perp\Psi  = \sqrt{2} \P_N^\perp\xi
\end{align}

\noi
and the nonlinear dynamics for the low frequency part $\P_N u_N$:
\begin{align}
\dt^2 \P_N u_N   + \dt \P_N u_N  +(1-\dx^2)  \P_N u_N 
+\g \P_N\big\{ \sin(\be \P_N u_N) \big\}= \sqrt{2} \P_N \xi 
\label{LowSdSG}
\end{align}

\noi
Also, the residual part $v_N = u_N - \Psi $ satisfies the following equation:
\begin{align}
\begin{cases}
\dt^2 v_N + \dt v_N +(1-\dx^2)v_N  +\g\P_N\Big\{ \sin\big(\be \P_N (v_N+\Psi) \big) \Big\}
=0\\
(v_N,\dt v_N)|_{t = 0}=(0,0). 
\end{cases}
\label{SNLW12}
\end{align}

\noi
Notice that we have imposed a structure of the solution under the form\footnote{The classical Bourgain \cite{BO96}, Da Prato-Debussche \cite{DP1, DP2} type argument like $u_N$=stochastic linear term+smoother term.} of
\begin{align}
\Phi_t^N(\u_0,\xi):&=u_N =S(t)\u_0 +\stick_{t}(\xi)+v_N, \label{DaPra}\\
\vec \Phi_t^N(\u_0,\xi):&=\big(   \Phi_t^N(\u_0,\xi),  \dt  \Phi_t^N(\u_0,\xi)   \big) \label{DaDD}
\end{align}

\noi
where $v_N$ also solves the following mild (Duhamel) formulation  
\begin{align}
v_{N}(t) = - \int_0^t \D(t-t')\g\P_{N}\Big\{ \sin\Big(\be\P_N \big(S(t')\u_0+\stick_{t'}(\xi)+v_{N}(t') \big)  \Big) \Big\} dt'.
\label{Pert}
\end{align}

\noi
Here, $v_N$ is expected to be smoother and hence falling into the scope of pathwise analysis (i.e.~pathwise well-posedness theory) once we have a control on the relevant stochastic terms. Consequently, our goal is to solve the perturbed equation \eqref{Pert} for $v_N$. In the realm of stochastic PDEs, a well-posedness approach based on the decomposition \eqref{DaPra} is commonly known as the Da Prato-Debussche trick \cite{DP1, DP2}. In the following we  obtain the following limiting equation in terms of $v$ 
\begin{align}
\begin{cases}
\dt^2 v + \dt v +(1-\dx^2)v  +\g\Big\{ \sin\big(\be  (v+S(t)\u_0+\stick_{t}(\xi)) \big) \Big\}
=0\\
(v,\dt v)|_{t = 0}=(0,0)
\label{SNLW13}
\end{cases}
\end{align}

\noi
and so define a solution $\Phi_t(\cdot,\xi)$ for the original equation \eqref{SdSG} in terms of the first-order expansion
\begin{align}
\Phi_t(\u_0,\xi):&=S(t)\u_0 +\stick_{t}(\xi)+v(t) \label{Dapra1},\\
\vec \Phi_t(\u_0,\xi):&=\big(   \Phi_t(\u_0,\xi),  \dt  \Phi_t(\u_0,\xi)   \big). \label{DaDD2}
\end{align}

In the upcoming content, $\g$ and $\be$ in \eqref{SdSG2} do not play any role  and thus we disregard them by setting $\g=1$ and $\be=1.$ We are now ready to prove the following propositions.

\begin{proposition}\label{PROP:GWP}
Let $N\in \N\cup\{\infty\}$\footnote{When $N=\infty$, it is understood as the limiting equation \eqref{SNLW13}.}. Then, the solution $v_N$ to \eqref{SNLW12} or \eqref{SNLW13} exists globally in time, almost surely, such that $\v_N=(v_N,\dt v_N) \in C(\R_{+}; \H^{1}(\T) )$.  
\end{proposition}

\begin{proof}
By studying the integral formulation \eqref{Pert}, a contraction argument in $C([0,T]; \H^1(\T))$ for small $T>0$ yields that the equations \eqref{SNLW12} and \eqref{SNLW13} have a unique solution up to time $T$. In taking $\H^1$-norm to \eqref{Pert}, thanks to the exponential term in \eqref{D1}\footnote{This comes from the damping term $\dt u$ in the equation \eqref{SdSG}}, we obtain an a priori bound on the $\H^1$-norm of a solution $\v_N$,  almost surely,
\begin{align*}
\sup_{t\in [0,\infty) }\|\v_N(t)\|_{\H^1}\le C,
\end{align*}

\noi
which implies that $\v_N(t)$ is almost surely defined on the half-line $[0,\infty)$ 
for any fixed $N\in \N\cup\{\infty\}$. This completes the proof of Proposition \ref{PROP:GWP}.

\end{proof}

Let $F:\H^{\frac 12-}(\T)\to \R$ be a bounded and measurable function. Given $N\in \N \cup \{\infty\} $, we define the family of bounded linear operators $\Pt{t}^N$ by
\begin{align}
F\mapsto \Pt{t}^NF(\u_0):&=\E\big[F(\vec \Phi_t^N(\u_0,\xi))  \big]
\label{MAR1}
\end{align}

\noi
where $\vec \Phi_t^N$ and $\vec \Phi_t$ are in \eqref{DaDD} and \eqref{DaDD2} and $\u_0=(u_0,v_0)\in \H^{\frac 12-}(\T)$. When $N =\infty $, it is understood as the (stochastic) flow map $\vec \Phi_t(\cdot,\xi)$ of the original equation \eqref{SdSG} and 
is denoted by
\begin{align}
F\mapsto \Pt{t}F(\u_0):&=\E\big[F(\vec \Phi_t(\u_0,\xi))  \big].
\label{MAR2}
\end{align} 

\noi
In particular, $\{\vec \Phi_t^N\}_{t\ge 0}$ is a Markov process and $\{\Pt{t}^N \}_{t\ge 0}$ is a Markov semigroup for every $N\in \N \cup \{\infty\}$.  We, however, point out that because of the actual structure of the solutions to \eqref{DaPra} or \eqref{Dapra1}, namely, first order expansion, it is not a priori clear if the operators in \eqref{MAR1} and \eqref{MAR2} define Markov semigroups on bounded functions $F:\H^{\frac 12-}(\T)\to \R$. By following the method presented in \cite[Section 4.1]{TW}, one can easily check that they satisfy the properties. 

We now turn to the proof of invariance of truncated Gibbs measure $\rhoo_N$ under the flow of \eqref{SdSG2}.  Given $N \in \N$, we first define the marginal probabilities measures $\muu_{N}$ and $\muu_{N}^\perp$ on $\P_N \H^{\frac 12-}(\T)$ and $\P_N^\perp \H^{\frac 12 -}(\T)$, respectively, as the induced probability measures under the following maps:
\begin{equation*}
\o \in \O \longmapsto (\P_N u(\o), \P_N v(\o))
\end{equation*}

\noi
for 
$\muu_{N}$
and
\begin{equation*}
\o \in \O \longmapsto (\P_N^\perp u(\o), \P_N^\perp v(\o))
\end{equation*}

\noi
for $\muu_{N}^\perp$, where $u$ and $v$ are as in \eqref{ranseries}.
Then, $\muu$ can be written as
\begin{align}
\muu = \muu_{N} \otimes \muu_{N}^\perp.
\label{X7}
\end{align}

\noi
Thanks to \eqref{TGibbs} and \eqref{X7}, we can write
\begin{align}
\rhoo_N =  \vec \nu_{N} \otimes \muu_{N}^\perp
\label{Pa07}
\end{align}

\noi
where $\vec \nu_N$ is as follows:
\begin{align*}
d \vec \nu_N(u,v)=Z_{N}^{-1}e^{\frac {\g}{\be }\int_\T \cos(\be \P_N u) dx }d\muu_{N}(u,v).
\end{align*}

\noi
We are now ready to prove the following proposition, based on the method \cite{GKOT22, ORTZVET}.

\begin{proposition}\label{PROP:Finvar}
Let $N\in \N$. The truncated Gibbs measure $\rhoo_N$ in \eqref{TGibbs} is invariant for the Markov process $\vec \Phi_t^N(\cdot,\xi)$ associated with the flow of \eqref{SdSG2}.
\end{proposition}

\begin{proof}
We first write
\begin{align*}
\vec\Phi^N_t(\cdot,\xi) = \big(\P_N(\Phi^N_t(\cdot,\xi), \dt\Phi^N_t(\cdot,\xi) )\big)+\big( 
\P_N^{\perp}(\Psi,\dt\Psi)(t) \big)  \text{ on }(E_N\times E_N)\oplus (E_N^{\perp}\times E_N^{\perp})
\end{align*} 

\noi
where $E_N=\text{span}\{ e^{inx}: |n|\le N \}$ and $E_N^{\perp}$ is the corresponding  orthogonal complement.  Notice that the evolution of the high frequency component $\P_N^\perp u_N = \P_N^\perp\Psi$ is described by
\begin{align}
\dt^2 \P_N^\perp\Psi + \dt  \P_N^\perp\Psi +(1-\dx^2) \P_N^\perp\Psi  = \sqrt{2} \P_N^\perp\xi,
\label{trhigh}
\end{align}

\noi
which is a linear dynamics and so it can be easily shown that the Gaussian measure $\muu_N^\perp$ is invariant under the flow of \eqref{trhigh} by following the argument in the previous subsection (Subsection \ref{SUBSEC:Stoconvol}). Regarding the low frequency part, $\P_N u_N=\P_N\Phi^N_t(\cdot,\xi)$ corresponds to the following nonlinear dynamics 
\begin{align}
\dt^2 \P_N u_N   + \dt \P_N u_N  +(1-\dx^2)  \P_N u_N 
+\g\P_N\big\{ \sin(\be\P_N u_N) \big\}= \sqrt{2} \P_N \xi  .
\label{SNLW11a}
\end{align}

\noi
By setting $(u_N, v_N) := (\P_N u_N, \dt \P_N u_N)$, the nonlinear dynamics \eqref{SNLW11a} for the low frequency part  can be written as the following (finite-dimensional) Ito SDE:
\begin{align}
\begin{split}
d  \begin{pmatrix}
u_N \\ v_N
\end{pmatrix}
& + \Bigg\{
\begin{pmatrix}
0  & -1\\
1-\dx^2 &  0
\end{pmatrix}
 \begin{pmatrix}
u_N \\ v_N
\end{pmatrix}
 +  
\begin{pmatrix}
0 \\ \g\P_N\big\{ \sin(\be \P_N u_N) \big\}
\end{pmatrix}
\Bigg\} dt \\
&   = 
  \begin{pmatrix}
0  \\ - v_N dt + \sqrt 2\P_N dW
\end{pmatrix} .
\end{split}
\label{SNLW16}
\end{align}

\noi
This exhibits that $\eqref{SNLW16} $ is composed of the deterministic Hamiltonian PDE  and the Langevin dynamics and so the generator $\L^N$ for \eqref{SNLW16} can be decomposed into two parts, denoted as $\L^N_1$ and $\L^N_2$, such that $\L^N = \L^N_1 + \L^N_2$, where $\L^N_1$ is the generator for the deterministic sine-Gordon equation  with the truncated nonlinearity:
\begin{align}
\begin{split}
d  \begin{pmatrix}
u_N \\ v_N
\end{pmatrix}
 + \Bigg\{
\begin{pmatrix}
0  & -1\\
1-\dx^2 &  0
\end{pmatrix}
 \begin{pmatrix}
u_N \\ v_N
\end{pmatrix}
 +  
\begin{pmatrix}
0 \\ \g \P_N\big\{ \sin(\be \P_N u_N) \big\}
\end{pmatrix}
\Bigg\} dt 
   = 0 
\end{split}
\label{SNLW17}
\end{align}

\noi
and $\L^N_2$ is the generator for the Ornstein-Uhlenbeck process in the momentum variable $v_N$:
\begin{align}
\begin{split}
d  \begin{pmatrix}
u_N \\ v_N
\end{pmatrix}
= 
\begin{pmatrix}
0  \\ - v_N dt + \sqrt 2\P_N dW
\end{pmatrix} .
\end{split}
\label{SNLW18}
\end{align}

\noi
Specifically, \eqref{SNLW17} represents a Hamiltonian PDE with the Hamiltonian given by:
\begin{align*}
H(u_N, v_N ) = \frac{1}{2}\int_{\T}\big(   |\dx u_N(x)|^2+ (u_N(x))^2\big) dx
+ \frac{1}{2}\int_{\T} (v_N(x))^2dx- \frac{\g}{\be}\int_{\T}\cos\big(\be u_N(x)\big)dx
\end{align*}
 
\noi
Then, thanks to  the conservation of the Hamiltonian 
$H(u_N, v_N )$ and Liouville's theorem
(on a finite-dimensional phase space $\P_N\H^{\frac 12-}(\T)$), it is established that $\vec \nu_N$ is invariant under the flow of \eqref{SNLW17}. Therefore, we have $(\L^N_1)^*\vec \nu_N = 0$. In terms of the second component $v_N=\dt  u_N$,  the measure $\vec \nu_N$ corresponds to the white noise measure $(\P_N)_{\#}\mu_0$ (namely, when projected onto the part $v_N$).
Hence, by exploiting the fact that the Langevin dynamics \eqref{SNLW18}
preserves the Gaussian measure, we obtain that $\vec \nu_N$ is also invariant under the dynamics of \eqref{SNLW18} (i.e.~$(\L^N_2)^*\vec \nu_N = 0$). By combining the results, we get $(\L^N)^*\vec \nu_N = (\L^N_1)^*\vec \nu_N +  (\L^N_2)^*\vec  \nu_N = 0$. This shows invariance of $\vec \nu_N$ under \eqref{SNLW16} and hence under \eqref{SNLW11a}.     

By combining \eqref{Pa07} and invariance of $\vec \nu_N$ and $\muu_N^\perp$ under \eqref{SNLW16} and \eqref{trhigh}, respectively,  we conclude that for any bounded and continuous function $F:\H^{\frac 12-}(\T) \to \R$ 
\begin{align*}
\int \E \Big[ F(\vec \Phi_t^N(\u_0,\xi))  \Big] d\rhoo_N (\u_0)=\int F(\u_0)   d\rhoo_N(\u_0),
\end{align*}

\noi
which shows that the truncated Gibbs measure $\rhoo_N$ in \eqref{TGibbs} is invariant under the folow of \eqref{SdSG2}.
\end{proof}

\begin{corollary}\label{COR:invar}
The Gibbs measure $\rhoo$ in \eqref{Gibbslim} is invariant for the Markov process $\vec \Phi_t(\cdot,\xi)$ associated with the flow of \eqref{SdSG}. 
\end{corollary}

\begin{proof}
For every $t \ge0$, $\o \in \O$, and $\u_0=(u_0,v_0)\in \H^{\frac 12-}(\T)$,  we have
\begin{align*}
\|\vec \Phi_t^N(\u_0,\xi(\o))-\vec \Phi_t(\u_0,\xi(\o)) \|_{\H^{\frac 12-}} \to 0
\end{align*}

\noi
as $N\to \infty$.
Let $F:\H^{\frac 12-}(\T) \to \R$ be a bounded and continuous function. Thanks to the dominated convergence theorem and Proposition \ref{PROP:Finvar}, we have
\begin{align*}
\int \E \Big[ F(\vec \Phi_t(\u_0,\xi))  \Big] d\rhoo (\u_0)&=\int \E\Big[F(\vec \Phi_t(\u_0,\xi))  \Big]e^{\frac {\g}{\be }\int_\T \cos(\be u_0)dx} d\muu( \u_0)\\
&=\lim_{N\to \infty }\int \E\Big[ F(\vec \Phi_t^N(\u_0,\xi))  \Big] e^{\frac {\g}{\be} \int_\T \cos(\be \P_N u_0 )dx} d\muu(\u_0)\\
&=\lim_{N\to \infty } \int F(\u_0) e^{\frac {\g}{\be}\int_\T \cos(\be \P_N u_0 )dx} d\muu( \u_0)\\
&=\int F(\u_0)  d\rhoo(\u_0).
\end{align*}

\noi
This completes the proof of Corollary \ref{COR:invar}.

\end{proof}

\section{Quantified ergodicity}
\label{SEC:exp}

\subsection{Exponential ergodicity}
\label{SUBSEC:EXPERG}

In this subsection, we present that the Gibbs measure $\rhoo$ is the unique invariant measure for the Markov process $\vec \Phi_t(\cdot,\xi)$. Moreover, its transition probabilities converge exponentially fast to the Gibbs measure in a type of 1-Wasserstein distance by assuming several ingredients whose proof will be present in subsequent subsections. Before going ahead, we first show that the dynamics \eqref{SdSG} does not satisfy the strong Feller property.

\begin{proposition}\label{PROP:fail}
Let $F(\u_0):= \ind_{\{\u_0 \in \H^{1/2-}\}}$. Then, we have  
\begin{align*}
\Pt{t}F(\u_0)=\E[F(\vec \Phi_t(\u_0,\xi))] = F(\u_0)
\end{align*}

\noi
for every $t \ge 0$. Therefore, $\Pt{t} F$ is a discontinuou function on $\H^{\frac 12-}$. 

\end{proposition}

\begin{proof}
We follow the argument in \cite[Proposition 5.1]{FT}. From \eqref{Dapra1} and \eqref{DaDD2}, we recall that $\Phi_{t}(\u_0; \xi)=S(t)\u_0+\stick_{t}(\xi)+v(t)$ and $\vec \Phi_{t}(\u_0; \xi)=(\Phi_{t}(\u_0; \xi), \dt \Phi_{t}(\u_0; \xi)) $.
Note that for every $t \ge 0$, $\vec \stick_t(\xi) \in  \H^{\frac 12-}$ and $\v(t) \in \H^1$ where $\vec \stick_t(\xi)=( \stick_t(\xi), \dt \stick_t(\xi))$ and $\v(t)=(v(t),\dt v(t))$. Hence, we have 
\begin{align*}
\big\{\vec \Phi_t(\u_0,\xi) \in \H^{\frac 12-} \big\}=\big\{ \vec S(t)\u_0 \in \H^{\frac 12-} \big\}.
\end{align*}

\noi
Since the linear wave propagator $\vec S(t)$ is invertible\footnote{
The linear propagator does not give any improvement of regularity (i.e.~persistence of regularity).} on $\H^s$, $s\in \R$, we obtain 
\begin{align*}
\big\{ \vec S(t)\u_0 \in \H^{\frac 12-} \big\}=\big\{  \u_0 \in \H^{\frac 12-} \big\}.
\end{align*}
 
\noi
Therefore, we have
\begin{align*}
\Pt{t}F(\u_0) = \E[F(\vec \Phi_t(\u_0,\xi))] = F(\u_0).
\end{align*}

\noi
Hence, $\Pt{t}F$ is a discontinuous function on $\H^{\frac 12-}$.
\end{proof}

In our context, due to the absence of the strong Feller property, it is not anticipated for strong convergence in the total variation distance to occur. Therefore, instead of proving the existence of a spectral gap in the total variation distance, we show that the Markov semigroup $\Pt{t}$ has a contraction property in a particular 1-Wasserstein distance (see \eqref{STARCONTRAC} below).
We set the following weighted metric on $\H^{\frac 12-}(\T)$ 
\begin{align*}
\varrho_\ld(\u_1^0,\u_2^0)=\inf_{\gamma:\u_1^0\to \u_2^0} \int_0^1  e^{\ld  \| \g(t) \|_{\H^{1/2-}}^2 }
\| \dot \g(t) \|_{\H^{\frac 12-}} dt
\end{align*}

\noi
where the infimum is taken over all Lipschitz continuous paths $\g: [0,1] \to \H^{\frac 12-}(\T)$ connecting $\u_1^0$ to $\u_2^0$. We consider the following distance
\begin{align}
d_\dl(\u_1^0, \u_2^0)=1\wedge \dl^{-1}\varrho_\ld(\u_1^0,\u_2^0)
\label{wedmet}
\end{align}

\noi
and its corresponding Wasserstein-1 distance
\begin{align*}
\W_{d_\dl}\big(\Pt{t}^*\dl_{\u_1^0}, \Pt{t}^*\dl_{\u_2^0} \big)&=\inf_{\pi } \intt_{\H^{\frac 12-}\times \H^{\frac 12-}}d_\dl(\vec {\bf u}, \vec {\bf v}) \pi(d \vec {\bf u},d \vec {\bf v}),
\end{align*}

\noi
where the infimum runs over all coupling $\pi \in \mathcal{C}(\Pt{t}^*\dl_{\u_1^0}, \Pt{t}^*\dl_{\u_2^0} )$. Thanks to the Monge–Kantorowitch duality, we have
\begin{align}
\W_{d_\dl}\big(\Pt{t}^*\dl_{\u_1^0}, \Pt{t}^*\dl_{\u_2^0} \big)=\sup_{[F]_{\Lip_{d_\dl}} \le 1}\bigg(\int F(y) d\Pt{t}^*\dl_{\u_1^0}(y)-\int F(y) d\Pt{t}^*\dl_{\u_2^0} (y)  \bigg)
\label{MGDU}
\end{align}

\noi
where
\begin{align*}
[F]_{\Lip_{d_\dl}}=\sup_{x\neq y} \frac{|F(x)-F(y)|}{d_\dl(x,y)}.
\end{align*}

\noi
In fact, we can further restrict the space of Lipschitz functions $F$ in the definition of \eqref{MGDU} as follows:
\begin{align}
\W_{d_\dl}\big(\Pt{t}^*\dl_{\u_1^0}, \Pt{t}^*\dl_{\u_2^0} \big)=\sup_{\substack{[F]_{\Lip_{d_\dl}} \le 1\\ \| F\|_{L^\infty}\le \frac 12}}\bigg(\int F(y) d\Pt{t}^*\dl_{\u_1^0}(y)-\int F(y) d\Pt{t}^*\dl_{\u_2^0} (y)  \bigg).
\label{MGDU1}
\end{align}

\noi
For the proof, see \cite[Lemma 5.2]{FT}. For convenience of notation, we denote by 
$\text{Lip}_{d_\dl}$ the set of $d_{\dl}$-Lipschitz continuous functions such that $[F]_{\Lip_{d_\dl}} \le 1$.

Based on the strategy introduced by \cite{HMS}, the geometric ergodicity follows from suitable large-time smoothing estimates (see Lemma \ref{LEM:gra} and \ref{LEM:gra1} below) for the Markov process, 
closely related to so-called asymptotic strong Feller property, together with a suitable irreducibility condition (see Lemma \ref{LEM:visit} below) for system \eqref{SdSG}. 
More precisely, we prove that (i) Wasserstein  distance $\W_{d_\dl}$ associated with the distance-like function $d_\dl$ \eqref{wedmet} is contracting for $\Pt{t}$ (Proposition \ref{PROP:contract}) and (ii) bounded sets are $d_\dl-$small (Proposition \ref{PROP:dsmall}) by exploiting the smoothing estimates and irreducibility condition. Then, the exponential convergence rate toward the Gibbs measure \eqref{Gibbs1} is obtained through a Wasserstein  distance.

\begin{proposition}[contracting distances]\label{PROP:contract}
There exists $\dl>0$ and $t^*>0$ such that the metric $d_\dl(\u_1^0,\u_2^0)=1\wedge \dl^{-1}\varrho_\eta(\u_1^0,\u_2^0)$  is contracting for the Markov semigroup $\Pt{t}$. More precisely, for every $t\ge t^*$, the following estimate
\begin{align*}
\W_{d_\dl}\big(\Pt{t}^*\dl_{\u_1^0}, \Pt{t}^*\dl_{\u_2^0} \big) \le \frac 12 d_\dl(\u_1^0, \u_2^0) 
\end{align*}

\noi
holds for every pair $\u_1^0, \u_2^0 \in \H^{\frac 12-}(\T)$ with  $d_\dl(\u_1^0, \u_2^0)<1$.

\end{proposition}

\begin{proposition}[$d_\dl$-small set]\label{PROP:dsmall}
Let $R>0$ and $\dl>0$. Then, there exists $t^*=t^*(R,\dl)>0$ such that the set $B_R$ is $d_\dl$-small for the Markov semigroup $\Pt{t}$. In other words, for every $t\ge t^*$ and
every $\u_1^0, \u_2^0 \in B_R$, we have that
\begin{align*}
\W_{d_\dl}\big(\Pt{t}^*\dl_{\u_1^0}, \Pt{t}^*\dl_{\u_2^0} \big) <1
\end{align*}

\noi
where
\begin{align*}
B_R=\Big\{\u_0 \in \H^{\frac 12-}: \| \u_0 \|_{\H^{\frac 12-} }\le R \Big\}.
\end{align*}

\end{proposition}

We point out that another key ingredient to prove Proposition \ref{PROP:contract} is to establish the existence of a Lyapunov function for \eqref{SdSG}. A Lyapunov function for a Markov semigroup $\{ \Pt{t} \}_{t\ge 0}$ over a Polish space $\H^{\frac 12-}$ is a function $V:\H^{\frac 12-}(\T)\to [0,\infty]$ such that $V$ is integrable with respect to $\Pt{t}^*\dl_{\u_0}$ for every $\u_0 \in \H^{\frac 12-}$ and $t\ge 0$. More precisely, there exist constants $C, \g$, and  $K>0$ such that the estimate
\begin{align}
\int V(y) \Pt{t}^*(\u_0, dy) \le C_V e^{-\g t}V(\u_0)+K_V
\label{Lya}
\end{align}

\noi
holds for every $\u_0 \in \H^{\frac 12-}$ and $t\ge 0$. Note that unlike parabolic equations, nonlinear damped  wave equations do not have a strong dissipation and so the existence of a Lyapunov function is not a priori clear. We first introduce the following propositions whose proof will be present in Subsections \ref{SUBSEC:Lya}.
\begin{lemma}\label{LEM:Lya}
There exist $\ld_0,\g, C, K>0$ such that for every $0<\ld \le \ld_0$, the Markov semigroup $\Pt{t}$ satisfies the estimate \eqref{Lya} 
\begin{align}
\E\Big[ \exp\big(\ld \|\vec \Phi_t(\u_0,\xi)\|_{\H^{\frac 12-}}^2 \big) \Big] \le C e^{-\g t} \exp(\ld \|\u_0\|_{\H^\frac 12-}^2)+K
\label{Lyapu0}
\end{align}

\noi
for every $\u_0\in \H^{\frac 12-}$ and $t\ge 0$.

\end{lemma}

By exploiting the result proven in \cite[Theorem 4.8]{HMS}, it is possible to show that 
if we establish Propositions \ref{PROP:contract}, \ref{PROP:dsmall}, and Lemma \ref{LEM:Lya}, then there exist $T>0$ and some constant $c>0$ such that
\begin{align}
\W_{ \wt d_\dl}\big( \Pt{t}^*\nu_1, \Pt{t}^*\nu_2 \big)\le e^{-c \lfloor \frac t{T} \rfloor } \W_{ \wt d_\dl}(\nu_1, \nu_2)
\label{STARCONTRAC}
\end{align}

\noi
for every $t\ge T$ and any $\nu_1, \nu_2 \in \mathcal{P}r(\H^{\frac 12-})$, where 
\begin{align}
\wt d_\dl(\u_0,\u_1)=\sqrt{d_\dl(\u_0,\u_1)(1+e^{\ld \|\u_0 \|^2_{\H^{ 1/2-}}}+e^{\ld \|\u_1 \|^2_{\H^{ 1/2-}}})}.
\label{WASREFdis}
\end{align}

\noi
The contraction property
\eqref{STARCONTRAC} implies that the Markov process $\vec \Phi_t(\u_0,\xi)$ has the unique invariant measure. In particular, by taking $\nu_1=\rhoo$ and $\nu_2=\dl_{\u_0}$ in \eqref{STARCONTRAC}, we can see that the Markov transition probabilities $\Pt{t}^*\dl_{\u_0}$ converge exponentially fast to the Gibbs measure $\rhoo$.

\subsection{Contracting distances for the Markov semigroup}
\label{SUBSEC:contract}
In this subsection, we establish the contracting distances (Proposition \ref{PROP:contract}) for the Markov semigroup $\{\Pt{t} \}_{t \ge 0}$. 
The main ingredient is to use the following gradient estimate (see Lemma \ref{LEM:gra} and \ref{LEM:gra1} below) whose proof is based on the Malliavin calculus. \noi
For further use, we note that for any $\vec h \in \H^{\frac 12-}$, we denote by $J_{t}\vec h$ the derivative of $\vec \Phi_t(\u_0,\xi)=(\Phi_t(\u_0,\xi), \dt \Phi_t(\u_0,\xi)) $ with respect to the initial condition $\u_0 \in \H^{\frac 12-}$ along the direction $\vec h$. Then, $J_{t}\vec h$ satisfies the following equation
\begin{align}
\frac d{dt} \pi_1 J_{t}\vec h&=\pi_2 J_{t}\vec h \notag \\
\frac d{dt} \pi_2 J_{t} \vec h&=-(1-\Dl)\pi_1 J_{t}\vec h-\pi_2 J_{t}\vec h-\cos(\Phi_t(\u_0,\xi))\pi_1 J_{t}\vec h \notag \\
J_{0}\vec h&=\vec h
\label{pert2}
\end{align}

\noi
where $\pi_1(u_1,u_2)=u_1$ and $\pi_2(u_1,u_2)=u_2$. We denote by $A_{t} v$ the Malliavin derivative of $\vec \Phi_t(\u_0,\xi)=(\Phi_t(\u_0,\xi), \dt \Phi_t(\u_0,\xi))$ with respect to $\xi$ along the perturbation $ v$. Then, $A_{t} v$ satisfies the following equation
\begin{align}
\frac d{dt} \pi_1 A_{t} v&=\pi_2 A_{t} v \notag \\
\frac d{dt} \pi_2 A_{t} v&=-(1-\Dl)\pi_1 A_{t}  v-\pi_2 A_{t}  v-\cos(\Phi_t(\u_0,\xi))\pi_1 A_{t} v + v\notag \\
A_{0} v&=\vec 0.
\label{pert1}
\end{align}

\noi
In summary, $J_{t} \vec h$ is the effect on $\vec \Phi_t(\u_,\xi)$ of an infinitesimal perturbation of the initial condition in the direction $\vec h$ and $A_{t}v $ is the effect on $\vec \Phi_t(\u_0,\xi)$ of an infinitesimal perturbation of the noise in the direction $v$. We now prove the gradient estimates, which exhibits suitable large-time smoothing effect for the Markov process, so-called asymptotic
strong Feller property.

\begin{lemma}\label{LEM:gra}
There exist constants $C>0$ and $\g>0$ such that for every $F\in \Lip_{d_\dl}$, the Markov semigroup $\Pt{t}$ satisfies the following bound
\begin{align}
\| D\Pt{t}F(\u_0) \|_{\H^{\frac 12-} }  \le C \Big(\|F\|_{L^\infty}+e^{-\g t} \sqrt{(\Pt{t}\| DF \|_{\H^{\frac 12-}}^2) (\u_0)  }  \Big)
\label{gra3}
\end{align}

\noi
for every $\u_0 \in \H^{\frac 12-}$.
\end{lemma}


\begin{proof}
By setting $\rhoo(t):=J_{0,t}\vec h-A_{0,t}\vec v$ and using \eqref{pert2}, \eqref{pert1}, and the Malliavin integration by part, we have
\begin{align}
\jb{D\Pt{t}F(\u_0),\vec h}_{\H^{\frac 12-}}&=\E\Big[ \jb{DF(\Phi_t(\u_0,\xi)), J_{0,t}\vec h }_{\H^{\frac 12-}} \Big] \notag \\
&=\E\Big[ \jb{DF(\Phi_t(\u_0,\xi)), \rhoo(t) }_{\H^{\frac 12-}} \Big]+\E\Big[ \jb{DF(\Phi_t(\u_0,\xi)), A_{t} v }_{\H^{\frac 12-}} \Big] \notag \\
&=\E\Big[ \jb{DF(\Phi_t(\u_0,\xi)), \rhoo(t) }_{\H^{\frac 12-}} \Big]+\E\Big[ F(\Phi_t(\u_0,\xi)) \int_0^t \jb{ v (s), d\xi(s)}_{H^{\frac 12-}} \Big],
\label{MalIP}
\end{align}

\noi
where the stochastic integral is interpreted in the Skorohod sense if $v$ is not adapted.
Notice that the residual error $\rhoo(t)$ satisfies the equation
\begin{align}
\frac d{dt} \pi_1 \rhoo(t)&=\pi_2 \rhoo(t) \notag \\
\frac d{dt} \pi_2 \rhoo(t)&=-(1-\Dl)\pi_1 \rhoo(t)-\pi_2 \rhoo(t)-\cos(\Phi_t(\u_0,\xi))\pi_1 \rhoo(t)- v  \notag \\
\rhoo(0)&=\vec h, 
\label{error}
\end{align}

\noi
which is a control problem, where $ v$ is the control to guarantee that
\begin{align*}
\E \bigg|\int_0^\infty \jb{ v (s), d\xi(s)}_{H^{\frac 12-}} \bigg|^2<\infty \qquad \text{and}  \qquad \E \Big[ \| \rhoo(t) \|_{\H^{\frac 12 -}}^2 \Big]\le e^{-c t}\| \vec h\|^2_{\H^{\frac 12- }}.
\end{align*} 

\noi
for some $c>0$. In terms of \eqref{error}, by choosing $v$ as follows
\begin{align*}
v(t)=-\cos(\Phi_t(\u_0,\xi) )\pi_1 \rhoo(t),
\end{align*}

\noi
which is adapted to the filtration $\{\F_t\}_{t\ge 0}$ induced by $\xi$, we obtain that
\begin{align}
 \| \rhoo(t) \|_{\H^\frac 12}^2 \les e^{-\frac t2} \|\vec h \|_{\H^{\frac 12-}}^2.
\label{conexp0}
\end{align}

\noi
Moreover, thanks to the It\^o isometry, we have
\begin{align}
\E \bigg|\int_0^t \jb{ v (s), d\xi(s)}_{H^{\frac 12-}} \bigg|^2&\le \int_0^t \E \| \pi_1 \rhoo(t) \|_{H^{\frac 12-}}^2 dt \notag \\
&\les \| h\|_{\H^{\frac 12-}}^2.
\label{Ito}
\end{align}

\noi
Therefore, by coming back to \eqref{MalIP} and combining \eqref{conexp0} and \eqref{Ito}, we obtain
\begin{align*}
\|D \Pt{t}F(\u_0) \|_{\H^{\frac 12-}}&\le \sup_{\| h\|_{\H^{\frac 12-} } \le 1 }\E \Big[ \| DF(\Phi_t(\u_0,\xi )) \|_{\H^{\frac 12-}}  \| \rho(t) \|_{\H^{\frac 12-}} \Big]\\
&\hphantom{XXX}+ \sup_{\| h\|_{\H^{\frac 12-} } \le 1 }\E\Bigg[\bigg| F(\Phi_t(\u_0,\xi)) \int_0^t \jb{\vec v (s), d\xi(s)}_{\H^{\frac 12-}} \bigg| \Bigg]\\
&\le \sqrt{(\Pt{t}\| DF \|_{\H^{\frac 12-}}^2) (\u_0)  }  \sup_{\| h\|_{\H^{\frac 12-} } \le 1 } \big(\E \| \rho(t) \|_{\H^{\frac 12-}}^2\big)^{\frac 12}\\
&\hphantom{XXX}+ \|F\|_{L^\infty}  \sup_{\| h\|_{\H^{\frac 12-} } \le 1 }  \E\Bigg[\bigg| \int_0^t \jb{\vec v (s), d\xi(s)}_{\H^{\frac 12-}} \bigg|^2 \Bigg]^{\frac 12}\\
&\le e^{-\frac t2}\sqrt{(\Pt{t}\| DF \|_{\H^{\frac 12-}}^2) (\u_0)  } +C\|F \|_{L^\infty},
\end{align*}

\noi
which completes the proof of Lemma \ref{LEM:gra}.

\end{proof}

\begin{lemma}\label{LEM:gra1}
There exist constants $C>0,\g>0$, and $\ld_0>0$ such that for every $F\in \Lip_{d_\dl}$, the Markov semigroup $\Pt{t}$ satisfies the following bound
\begin{align}
\| D\Pt{t}F(\u_0) \|_{\H^{\frac 12-} }  \le C \exp\big(\ld \|\u_0\|_{\H^{\frac 12-}}^2    \big) \Big(\|F\|_{L^\infty}+\dl^{-1} e^{-\g t}   \Big)
\label{gra4}
\end{align}

\noi
for every $0<\ld\le \ld_0$ and $\u_0 \in \H^{\frac 12-}$.
\end{lemma}

\begin{proof}

We consider the term $\sqrt{(\Pt{t}\| DF \|_{\H^{\frac 12-}}^2) (\u_0)  }=\sqrt{ \E  \|DF(\Phi_t(\u_0,\xi)) \|_{\H^{\frac 12-}}^2   }$ in \eqref{gra3}. By choosing the specific path $\g^*$
\begin{align*}
\g^*(t)=t\u_1^0+(1-t)\u_2^0,
\end{align*}

\noi
we have that for any $F\in \Lip_{d_\dl}$
\begin{align*}
|F(\u_1^0)-F(\u_2^0)| \le d_{\dl}(\u_1^0, \u_2^0)&\le \frac 1\dl \varrho_{\ld}(\u_1^0,\u_2^0)\\
&\le  \frac 1\dl \int_0^1 \exp\big(\ld \| \g^*(t) \|_{\H^{\frac 12-}}^2 \big) \| \dot \g^*(t) \|_{\H^{\frac 12-}} dt,
\end{align*}

\noi
which implies
\begin{align}
\|DF(\u_0) \|_{\H^{\frac 12-}}= \sup_{ \|h\|_{\H^{\frac 12-}}  \le 1 }   \big| \jb{ DF(\u_0),h  }_{\H^{\frac 12-}} \big|\le \frac 1\dl \exp\big(\ld \| \u_0 \|_{\H^{\frac 12-}}^2 \big)
\label{grah}
\end{align}

\noi
Hence, thanks to \eqref{grah} and Lemma \ref{LEM:Lya}, we have
\begin{align}
\E \Big[  \|DF(\Phi_t(\u_0,\xi)) \|_{\H^{\frac 12-}}^2  \Big] 
&\le \dl^{-2}\E\Big[ \exp\big(\ld \| \Phi_t(\u_0,\xi) \|_{\H^{\frac 12-}}^2 \big)   \Big]  \notag \\
& \le  \dl^{-2} \Big( C  \exp\big(\ld \|\u_0\|_{\H^{\frac 12-}}^2 \big) +K \Big)
\label{gradrho}
\end{align}

\noi
for some $0<\ld \ll 1$. By putting \eqref{gradrho} into \eqref{gra3}, we obtain
\begin{align*}
\| D\Pt{t}F(\u_0) \|_{\H^{\frac 12-} }  &\le C \Big(\|F\|_{L^\infty}+e^{-\g t} \sqrt{(\Pt{t}\| DF \|_{\H^{\frac 12-}}^2) (\u_0)  }  \Big)\\
&\le  C \exp\big(\ld \| \u_0\|_{\H^{\frac 12-}}^2    \big)\Big(\|F\|_{L^\infty}+\dl^{-1} e^{-\g t}   \Big),
\end{align*}

\noi
which completes the proof of \eqref{gra4}.
\end{proof}

Now, we are ready to prove that the distance $\W_{d_\dl}$ is contracting for the Markov semigroup $\Pt{t}$ (namely, Proposition \ref{PROP:contract}): there exist $\dl>0$ and $t^*>0$ such that 
\begin{align*}
\W_{d_\dl}\big(\Pt{t}^*\dl_{\u_1^0}, \Pt{t}^*\dl_{\u_2^0} \big) \le \frac 12 d_\dl(\u_1^0, \u_2^0)
\end{align*}

\noi
for every pair $\u_1^0, \u_2^0 \in \H^{\frac 12-}(\T)$ with  $d_\dl(\u_1^0, \u_2^0)<1$ and $t \ge t^*$.

\begin{proof}[Proof of Proposition \ref{PROP:contract}]
Thanks to the Monge-Kantorovich duality in \eqref{MGDU1}, it suffices to prove that there exists $t^*>0$ and $\dl>0$ such that the bound
\begin{align}
\W_{d_\dl}\big(\Pt{t}^*\dl_{\u_1^0}, \Pt{t}^*\dl_{\u_2^0} \big)&=\sup_{\substack{[F]_{\Lip_{d_\dl}} \le 1\\ \| F\|_{L^\infty} \le 1 }} 
\bigg(\int F(y) d\Pt{t}^*\dl_{\u_1^0}(y)-\int F(y) d\Pt{t}^*\dl_{\u_2^0}(y) \bigg) \le \frac 1{2\dl} \varrho_\ld(\u_1^0,\u_2^0)
\label{cont}
\end{align}


\noi
holds for every $t\ge t^*$ when $d_\dl(\u_1^0, \u_2^0)<1$. From Lemma \ref{LEM:gra1}, we have that for every $\u_0,  h\in \H^{\frac 12-}$,
\begin{align*}
\big|\jb{D\Pt{t}F(\u_0), h }_{\H^{\frac 12-}} \big| \le C\exp\big(\ld  \|\u_0\|_{\H^{\frac 12-}}^2 \big)  \big(\|F\|_{L^\infty}+\dl^{-1}e^{-\g t}  \big)\| h \|_{\H^{\frac 12-}}.
\end{align*}

\noi
By choosing $\dl \ll 1$ and $t\ge t^*\gg 1$, we have
\begin{align}
\big|\jb{D\Pt{t}F(\u_0), h }_{\H^{\frac 12-}} \big|&\le C\dl^{-1}(e^{-\g t}+\dl) \exp\big(\ld \|\u_0\|_{\H^{\frac 12-}}^2 \big) \| h\|_{\H^{\frac 12-}} \notag \\
&\le \frac 12 \dl^{-1}  \exp\big(\ld \|\u_0\|_{\H^{\frac 12-}}^2 \big)  \| h\|_{\H^{\frac 12-}}.
\label{gra2}
\end{align}

\noi
For any Lipschitz continuous paths $\g: [0,1] \to \H^{\frac 12-}(\T)$ connecting $\u_1^0$ to $\u_2^0$, it follows from \eqref{gra2} that we have
\begin{align}
\Pt{t}F(\u_1^0)- \Pt{t}F(\u_2^0)&=\int_0^1  \jb{D\Pt{t}F(\g(t)), \dot\g(t)  }_{\H^{\frac 12-}} dt \notag \\
&\le \frac 1{2\dl}\int_0^1 \exp\big(\ld \|\g(t)\|_{\H^{\frac 12-}}^2 \big) \|\dot \g(t) \|_{\H^{\frac 12-} } dt.
\label{diffgra}
\end{align}

\noi
Since the above estimate \eqref{diffgra} holds for any Lipschitz continuous paths connecting $\u_1^0$ to $\u_2^0$, we have
\begin{align*}
|\Pt{t}F(\u_1^0)- \Pt{t}F(\u_2^0) |\le \frac 1{2\dl} \inf_{\gamma: \u_1^0\to \u_2^0} \int_0^1 e^{\ld \| \g(t)\|_{\H^{1/2-}}^2  } \| \dot \g(t) \|_{\H^{\frac 12-}} dt = \frac 1{2\dl} \varrho_\ld(\u_1^0,\u_2^0). 
\end{align*}

\noi
This completes the proof of the required bound \eqref{cont}.

\end{proof}


\subsection{On irreducibility conditions and the small sets}
In this subsection, we prove that every bounded set $B_R$ in \eqref{BR} is $d_\dl$-small for the Markov semigroup $\Pt{t}$. First, we state the following proposition which shows that the solution $\vec \Phi_t(\u_0,\xi)$ has always a chance to be able to visit any small ball centered at the origin.

\begin{lemma}\label{LEM:visit}
Let $R>0$ and $r>0$. Then, there exists $T^*=T^*(R,r)$ sufficiently large such that for every $t \ge T^*$, we have
\begin{align*}
\inf_{\u_0 \in B_R} \PP\Big\{ \| \vec \Phi_t(\u_0,\xi) \|_{\H^{\frac 12-}} <r      \Big\}=c>0
\end{align*}

\noi
for some $c=c(t,r)$, where 
\begin{align}
B_R=\Big\{\u_0 \in \H^{\frac 12-}: \| \u_0 \|_{\H^{\frac 12-} }\le R \Big\}.
\label{BR}
\end{align}

\end{lemma}

We present the proof of Lemma \ref{LEM:visit} at the end of this subsection. We first prove that $B_R$ is $d_\dl$-small set by assuming Lemma \ref{LEM:visit}.
\begin{proof}[Proof of Proposition \ref{PROP:dsmall}]
Let $\u_1^0, \u_2^0 \in B_R$. By setting $ B_\eps:=\{\u_0\in \H^{\frac 12-}: \|\u_0\|_{\H^{\frac 12-}} \le \eps \}$, we note that
\begin{align*}
\W_{d_\dl}\big(\Pt{t}^*\dl_{\u_1^0}, \Pt{t}^*\dl_{\u_2^0} \big)&=\inf_{\pi } \iintt_{\H^{\frac 12-}\times \H^{\frac 12-}}d_\dl(\vec {\bf u},\vec {\bf v}) \pi(d\vec {\bf u} , d\vec {\bf v} )\\
& \le \E\big[d_\dl(X,Y) \big]\\
& \le \E \Big[\big( 1 \wedge \dl^{-1}\varrho_\ld(X,Y) \big)\ind_{ \{ X\in B_\eps, Y \in B_\eps \} } \Big] +\E \big[\ind_{ \{ X\in B_\eps, Y \in B_\eps \}^c } \big], 
\end{align*}

\noi
where the infimum runs over all $\pi \in \mathcal{C}(\Pt{t}^*\dl_{\u_1^0}, \Pt{t}^*\dl_{\u_2^0} )$ and in the second step we choose the coupling $\pi$ such that $X$ and $Y$ are independent with $\Law_\PP(X)=\Pt{t}^*\dl_{\u_1^0}$ and $\Law_\PP(Y)=\Pt{t}^*\dl_{\u_2^0}$, respectively. By choosing the path $\g^*(t)=tX+(1-t)Y$, we have
\begin{align*}
\varrho_\ld(X,Y)&=\inf_{\gamma:X\to Y} \int_0^1 \exp\Big(\ld \|\g(t)\|_{\H^{\frac 12-}}^2 \Big) \| \dot \g(t) \|_{\H^{\frac 12-}} dt\\
&\le \exp\big(2 \ld (\| X \|^2_{\H^{1/2-}}+ \| Y \|^2_{\H^{1/2-}})  \big)\|X-Y \|_{\H^{\frac 12-}}.
\end{align*}

\noi
By choosing $\eps=\eps(\dl)$ sufficiently small, we have
\begin{align*}
\E \Big[\big( 1 \wedge \dl^{-1}\varrho_\ld(X,Y) \big)\ind_{ \{ X\in B_\eps, Y \in B_\eps \} } \Big] &\le 2\dl^{-1}e^{4\ld \eps^2}\eps \PP\big\{X\in B_\eps, Y\in B_\eps   \big\}\\
&\le \frac 12 \PP\big\{X\in B_\eps, Y\in B_\eps  \big\}.
\end{align*}

\noi
Therefore, we have that
\begin{align}
\W_{d_\dl}\big(\Pt{t}^*\dl_{\u_1^0}, \Pt{t}^*\dl_{\u_2^0} \big)&\le 1-\frac 12 \PP\big\{X\in B_\eps, Y\in B_\eps   \big\}   \notag \\
&=1-\frac 12 \PP\big\{X\in B_\eps \big \} \PP\big\{Y\in B_\eps \big \},
\label{indep}
\end{align}

\noi
where we used the fact that the coupling was chosen such that $X$ and $Y$ are independent.
It follows from Lemma \ref{LEM:visit} that there exists $T^*=T^*(R,\eps)>0$ such that for every $t \ge T^*$, we have
\begin{align}
\inf_{\u_0\in B_R}\PP\Big\{  \|\vec \Phi_t(\u_0,\xi)  \|_{\H^{\frac 12-}} \le \eps \Big\}>0
\label{visit1}
\end{align}

\noi
Hence, from \eqref{indep} and \eqref{visit1}, we obtain that
\begin{align*}
\sup_{\u_1^0, \u_2^0 \in B_R}\W_{d_\dl}\big(\Pt{t}^*\dl_{\u_1^0}, \Pt{t}^*\dl_{\u_2^0} \big) <1
\end{align*}

\noi
for any fixed $t \ge T^*$. This completes the proof of Proposition \ref{PROP:dsmall}.

\end{proof}

Now we present the proof of Lemma \ref{LEM:visit}. We recall the stochastic linear damped wave equation
\begin{align}
\partial_t \begin{pmatrix} u \\ \dt u \end{pmatrix}=
-\begin{pmatrix}0 & -1 \\ 1-\dx^2 & 1\end{pmatrix} \begin{pmatrix}u \\ \dt u \end{pmatrix} + \begin{pmatrix} 0 \\ {\sqrt2 \xi} \end{pmatrix}.
\label{error3}
\end{align}

\noi
We denote by $\vec L_t(\u_0,\xi)$ the flow of linear equation \eqref{error3}
\begin{align}
L_t(\u_0,\xi)&=S(t)\u_0+\stick_{t}(\xi) \notag \\
\vec L_t(\u_0,\xi)&=( L_t(\u_0,\xi), \dt  L_t(\u_0,\xi)).
\label{URISLIN}
\end{align}

\noi
Before we present the proof of Lemma \ref{LEM:visit}, we state the following lemma which corresponds to Lemma \ref{LEM:visit} for the linear flow \eqref{error3}. 
\begin{lemma}\label{LEM:visit1}
Let $R>0$ and $r>0$. Then, there exists $T^*=T^*(R,r)$ sufficiently large such that for any fixed $t \ge T^*$, we have
\begin{align*}
\inf_{\u_0 \in B_R} \PP\Big\{ \| \vec L_t(\u_0,\xi) \|_{\H^{\frac 12-}}  <r      \Big\}\ge c>0
\end{align*}

\noi
for some constant $c=c(t,r)$.
\end{lemma}

\begin{proof}[Proof of Lemma \ref{LEM:visit1}]
In the absence of the noise $\xi$ and nonlinear term, the deterministic linear flow is forced to go to zero exponentially fast. More precisely, for any $s\in \R$, we have
\begin{align*}
\|\vec S(t)  \u_0 \|_{\H^{s}} \le e^{-\frac t2}\| \u_0\|_{\H^{s}}.
\end{align*}

\noi
Consequently, we have $ \|\vec  L_t(\u_0,\xi) \|_{\H^{\frac 12-}}  \le e^{-\frac t2}\| \u_0\|_{\H^{\frac 12-}}+\|\vec \stick_{t}(\xi) \|_{\H^{\frac 12-}}$, which implies that by taking $T^*=T^*(R,r)>0$,
\begin{align*}
\inf_{\| \u_0\|_{\H^{\frac 12-}} \le R}\PP\Big\{ \| \vec L_t(\u_0,\xi) \|_{\H^{\frac 12-}}  \le r   \Big\} \ge \PP\Big\{  \|\vec \stick_{t}(\xi) \|_{\H^{\frac 12-}} \le \frac r2   \Big \}
\end{align*}

\noi
for any $t\ge T^*$, where $\vec \stick_{t}(\xi)=(\stick_{t}(\xi), \dt \stick_{t}(\xi))$. Since the law of $\vec \stick_{t}(\xi) $ is (non-degenerate) Gaussian, its law is full\footnote{Namely, the support of $\stick_{t}(\xi)$ is the whole of $\H^{\frac 12-}$.} in $\H^{\frac 12-}$. Therefore, we prove Lemma \ref{LEM:visit1}.

\end{proof}

We conclude this subsection by presenting the proof of Lemma \ref{LEM:visit}.

\begin{proof}[Proof of Lemma \ref{LEM:visit}]
We  define the measure $\mathbb{Q}$ whose Radon-Nikodym derivative 
with respect to $\PP$ is given by the following stochastic exponential:
\begin{align*}
\frac{d\mathbb{Q}}{d\PP}=\exp\bigg\{-\frac 12 \int_0^t \| h(t')\|_{L^2_x}^2 dt'+\int_0^t \jb{h(t'),d\xi(t')}_{L^2_x} \bigg \}=\EE(h(t))
\end{align*}

\noi
where $h(t):=\frac 1{\sqrt{2}}\sin\big(\Phi_t(\u_0,\xi)  \big)$. Then, the process $h(t)$ is adapted to the filtration $\{\F_t\}_{t\ge 0}$ induced by $\xi$
and satisfies the Novikov condition
\begin{align*}
\E\bigg[\exp\bigg\{ {\frac 12\int_0^t \|h(t') \|_{L^2_x}^2 dt'  }\bigg\} \bigg]\le e^{\frac t4}<\infty
\end{align*}

\noi
for any fixed $t \ge 0$, which implies that $d\mathbb{Q}=\EE(h(t))d\PP$ is a probability measure. Thanks to the Girsanov theorem, we have
\begin{align*}
\Law_\PP(\vec \Phi_t(\u_0,\xi))=\Law_\mathbb{Q}(\vec L(t)\u_0),
\end{align*}

\noi
where $\vec L(t)\u_0=\vec S(t)\u_0+\vec \stick_{t}(\xi)$. Hence, we have
\begin{align}
\PP\Big\{  \| \vec \Phi_t(\u_0,\xi) \|_{\H^{\frac 12-}}  <r   \Big\}&=\E\bigg[\EE(t)\ind_{ \big\{ \| \vec L_t(\u_0,\xi) \|_{\H^{\frac 12-}}  <r    \big\} }   \bigg]   \notag \\
&\ge C\PP\Big\{ \EE(t) >C, \; \| \vec L_t(\u_0,\xi) \|_{\H^{\frac 12-}}  <r\Big\}.
\label{cut1}
\end{align}

\noi
Note that 
\begin{align}
\PP\Big\{ \| \vec L_t(\u_0,\xi) \|_{\H^{\frac 12-}}  <r    \Big\}
&\le  \PP\Big\{  \EE(t)>C, \;  \| \vec L_t(\u_0,\xi) \|_{\H^{\frac 12-}}  <r \Big\}+\PP\big \{ \EE(t) \le C \big \}.
\label{cut2}
\end{align}

\noi
By combining \eqref{cut1} and \eqref{cut2}, we obtain
\begin{align*}
\PP\Big\{ \| \vec \Phi_t(\u_0,\xi) \|_{\H^{\frac 12-}}  <r   \Big\} \ge  C\PP\Big\{  \| \vec  L_t(\u_0,\xi) \|_{\H^{\frac 12-}}  <r    \Big\} - C\PP\big\{ \EE(t) \le C \big\},
\end{align*}

\noi
which implies that
\begin{align}
\inf_{\u_0\in B_R}\PP\Big\{ \| \vec \Phi_t(\u_0,\xi) \|_{\H^{\frac 12-}}  <r   \Big\} &\ge  C\inf_{\u_0 \in B_R}\PP\Big\{  \|\vec  L_t(\u_0,\xi) \|_{\H^{\frac 12-}}  <r    \Big\}  - C\sup_{\u_0 \in B_R}\PP\big\{ \EE(t) \le C \big\}.
\label{thrprob}
\end{align}

\noi
Thanks to Lemma \ref{LEM:visit}, the first term in \eqref{thrprob} can be bounded as follows: there exists $T^*(R,r)>0$ such that
\begin{align}
\inf_{\u_0 \in B_R}\PP\Big\{  \| \vec L_t(\u_0,\xi) \|_{\H^{\frac 12-}}  <r    \Big\} \ge \eta>0
\label{visit0}
\end{align} 

\noi
for any fixed $t \ge T^*$ and some $\eta>0$. Therefore, for any fixed $t\ge T^*$, it suffices to prove that by taking $C$ sufficiently small,
\begin{align}
\sup_{\u_0\in B_R}\PP\big\{ \EE(t) \le C \big\} < \eps
\label{twoprob}
\end{align}

\noi
for some $0<\eps \ll 1$ with $0<\eps \ll \eta$. From Markov inequality and It\^o isometry, we have
\begin{align*}
\PP\{ \EE(t) \le C \}&=\PP\bigg\{ \frac 12 \int_0^t \| h(t')\|_{L^2_x}^2 dt'-\int_0^t \jb{h(t'),d\xi(t')}_{L^2_x} \ge \log(C^{-1})  \bigg\}\\
&\le \frac 1{\log(C^{-1})} \Bigg\{   \E \bigg[\frac 12\int_0^t \| h(t')\|_{L^2_x}^2 dt' \bigg]+ \E \bigg[ \Big| \int_0^t \jb{h(t'),d\xi(t')}_{L^2_x} \Big| \bigg]     \Bigg\}\\
& \le \frac 1{\log(C^{-1})}\max\{ {t,t^{\frac 12}} \}
\end{align*}

\noi
Notice that the constant in the above estimate does not depend on $\u_0 \in B_R$. 
Hence, by taking $0<C=C(t)\ll 1$, we have
\begin{align}
\sup_{\u_0\in B_R}\PP\big\{ \EE(t) \le C \big\} < \eps.
\label{epsprob}
\end{align}

\noi
By combining \eqref{thrprob}, \eqref{visit0}, and \eqref{twoprob}, we obtain
\begin{align*}
\inf_{\u_0\in B_R}\PP\Big\{  \| \vec \Phi_t(\u_0,\xi) \|_{\H^{\frac 12-}}  <r   \Big\} \ge C(\eta-\eps)>0
\end{align*}

\noi
for any fixed $t\ge T^*$. This completes the proof of Lemma \ref{LEM:visit}.

\end{proof}

\subsection{A Lyapunov function for the Markov semigroup}
\label{SUBSEC:Lya}
We prove the existence of a Lyapunov function for the system \eqref{SdSG}. In contrast to parabolic equations, damped nonlinear wave equations do not have a strong dissipation, which raises uncertainty about the a priori existence of a Lyapunov function.
For example, by taking the linear flow $\vec L_t(\u_0,\xi)$ in \eqref{URISLIN}, we have   
\begin{align}
\frac{d}{dt}  V(\vec L_t(\u_0,\xi) )=-\frac 12 \|  \dt L_t(\u_0,\xi)(t)    \|_{H^{-\frac 12-}}^2
\label{VV1}
\end{align}

\noi
where $V(\u)=\frac 12\| \u\|_{\H^{\frac 12-}}^2$ and $\u=(u,v)$. This shows the absence\footnote{Namely, \eqref{VV1} shows only the damping effect in the momentum direction.} of the smoothing effect in the first component of $\u=(u,v)$. In the following, we exploit a modified energy technique to exhibit a hidden smoothing effect in the  first component of $\u$ (see \eqref{gen} and Remark \ref{REM:MODEN} below). 

\begin{proof}[Proof of Lemma \ref{LEM:Lya}]
We define
\begin{align*}
V_{\text{mod}}(\u)=V(\u)+M(\u) 
\end{align*}

\noi
where $V(\u)=\frac 12\| \u\|_{\H^{\frac 12-}}^2$ is as in \eqref{VV1} and $M(\u)$ is the energy correction term as follows:
\begin{align*}
M(\u)=\frac 14\| u\|_{H^{-\frac 12-}}^2+\frac 12 \jb{u, \dt u}_{H^{-\frac 12-}}.
\end{align*}

\noi
By Cauchy-Schwarz and Young's inequality, we have
\begin{align}
\frac 12 V(\u) \le  V_{\text{mod}}(\u) \le 2 V(\u).
\label{equiv}
\end{align}

\noi
Namely, the modified energy $V_{\text{mod}}(\u)$ is comparable with the original energy $V(\u)$. Hence, it suffices to show \eqref{Lyapu0} with $ V_{\text{mod}}(\u)$.
Let us recall the infinitesimal generator of \eqref{LowSdSG} in terms of the coordinates $a_n=\ft u(n)$ and $b_n=\ft v(n)$ (see \cite[Section 8]{CEGLA} or \cite[(2.6)]{Ngu}):
\begin{align}
&\L_Nf(a_0,...,a_{N},b_0,...,b_N) \notag \\
&= \sum_{n=0}^{N}b_n\partial_{a_n}f-\jb{n}^2a_n\partial_{b_n}f - b_n\partial_{b_n}f+\partial_{b_n}^2f- \F\bigg\{\sin\Big(\sum_{m=0}^N a_me^{imx}\Big) \bigg\}_n \partial_{b_n}f 
\label{gen0}
\end{align}

\noi
where $\F$ is the Fourier transform. Let $\ld >0$. We define 
\begin{align}
G_\ld(\u)=\exp\big( \ld  V_{\text{mod}}(\u)  \big).
\label{exp0}
\end{align}

\noi
Then, we have
\begin{align}
\L_N G_\ld(\P_N \u)=\ld G_\ld( \P_N \u ) \L_N  V_{\text{mod}}(\P_N \u)+\ld^2 G_\ld(\P_N \u) \sum_{n=0}^N |\partial_{b_n}    V_{\text{mod}}(\P_N \u) |^2. 
\label{gen00}
\end{align}

\noi
Regarding the first term in \eqref{gen00}, from Cauchy-Schwarz and Young's inequality, we obtain
\begin{align}
&\L_N  V_{\text{mod}}(\P_N\u) \notag \\
&=-\frac 12\sum_{n=0}^{N} \big( \jb{n}^{1-}|a_n|^2+ \jb{n}^{-1-}|b_n|^2-2\jb{n}^{-1-}\big)-\sum_{n=0}^N \big(\jb{n}^{-1-}(b_n+ {a_n}2^{-1}) \big) \F\bigg\{\sin\Big(\sum_{m=0}^N a_me^{imx}\Big) \bigg\}_n \notag \\
&\le-\frac 14\| \P_N u\|_{H^{\frac 12-}}^2 -\frac 14\| \P_N v \|_{H^{-\frac 12-}}^2+C
\label{gen}
\end{align}

\noi
where $\u=(u,v)$ and $C$ is a positive constant. As for the second term in \eqref{gen00}, note that 
\begin{align}
\sum_{n=0}^N |\partial_{b_n}    V_{\text{mod}}(\P_N \u)|^2 = \sum_{n=0}^N  \Big| \jb{n}^{-1-} (b_n+2^{-1} a_n ) \Big|^2 \le \frac 12\| \P_N u \|_{H^{-1-}}^2+ 2\| \P_N v \|_{H^{ -1-}}^2
\label{bn}
\end{align}

\noi
By combining \eqref{gen00}, \eqref{gen}, \eqref{bn}, and \eqref{equiv}, we have
\begin{align}
&\L_NG_\ld( \P_N \vec \Phi_t^N(\u_0,\xi) )\notag \\
&\le-\frac \ld4  G_\ld(\P_N\vec \Phi_t^N(\u_0,\xi) )\big( \| \pi_1  \P_N \vec \Phi_t^N(\u_0,\xi) \|_{H^{\frac 12-}}^2 +\| \pi_2 \P_N  \vec \Phi_t^N(\u_0,\xi) \|_{H^{-\frac 12-}}^2-4C   \big) \notag \\
&\hphantom{X}+2\ld^2  G_\ld(\P_N \vec \Phi_t^N(\u_0,\xi) ) \big( \| \pi_1 \P_N \vec \Phi_t^N(\u_0,\xi) \|_{H^{\frac 12-}}^2 +\| \pi_2 \P_N \vec \Phi_t^N(\u_0,\xi) \|_{H^{-\frac 12-}}^2   \big)    \notag  \\
&\le -\frac \ld4 G_\ld (\P_N \vec \Phi_t^N(\u_0,\xi)   )  \big( V_{\text{mod}}(\P_N \vec \Phi_t^N(\u_0,\xi)) -4 C   \big) +2\ld^2  G_\ld(\P_N \vec \Phi_t^N(\u_0,\xi) ) V_{\text{mod}}(\P_N \vec \Phi_t^N(\u_0,\xi))
\label{Dyn}
\end{align}

\noi
for any $t\ge 0$ and $N\in \N$, where $\vec \Phi_t^N(\cdot,\xi)$ is the flow map in \eqref{SdSG2}.
By taking $0<\ld \ll 1$ sufficiently small so that $\ld\ll \ld^2$, we have
\begin{align}
\eqref{Dyn}\le -\frac \ld8  G_\ld (\P_N \vec \Phi_t^N(\u_0,\xi)) V_{\text{mod}}(\P_N \vec\Phi_t^N(\u_0,\xi)) +C\ld  G_\ld (\P_N \vec \Phi_t^N(\u_0,\xi)   ).
\label{DDRI}
\end{align}

\noi
Then, from \eqref{DDRI}, \eqref{exp0} and the simple inequality $re^{ r}+1 \ge e^{ r}$ for any $r\ge 0$, we obtain that there exist $0<\ld_0\ll 1$, and $\g>0$ such that for any $0<\ld\le \ld_0$, we have   
\begin{align}
\eqref{Dyn}&\le -\frac 18 G_\ld (\P_N\Phi_t^N(\u_0,\xi))+\frac 18+C\ld  G_\ld (\P_N\Phi_t^N(\u_0,\xi)   ) \notag \\
&\le  -\g G_\ld (\P_N\Phi_t^N(\u_0,\xi))+\frac 18.
\label{DDRI1}
\end{align}

\noi
By applying the It\^o formula to $G_\ld(\P_N\vec \Phi_t^N(\u_0,\xi))$ and taking the expectation in both side of \eqref{DDRI1}, we have
\begin{align*}
\frac d{dt} \E\big[ G_\ld(\P_N \vec \Phi_t^N(\u_0,\xi) ) \big]=\E \big[ \L_NG_\ld(\P_N \vec \Phi_t^N(\u_0,\xi) ) \big] \le -\g \E\big[G_\ld(\P_N \vec \Phi_t^N(\u_0,\xi)) \big]+\frac 18. 
\end{align*}

\noi
From the Grownwall's inequality, we obtain that there exists $C>0$ such that
\begin{align}
\E\big[  G_\ld(\P_N \vec \Phi_t^N(\u_0,\xi) ) \big]\le Ce^{-\g t}G_\ld(\P_N \u_0)+C
\label{KS2}
\end{align}  

\noi
for any $t\in \R$ and $N\in \N$. It follows from \eqref{KS2}, \eqref{exp0}, and  \eqref{equiv} that 
\begin{align*}
\E\Big[ \exp\big(\ld \| \P_N\vec \Phi_t^N(\u_0,\xi)\|_{\H^{\frac 12-}}^2   \big)  \Big] \le C e^{-\g t} \exp(\ld \|\P_N \u_0\|_{\H^{\frac 12-} }^2  )+K.
\end{align*}

\noi
for some constant $C,K>0$. By using the Lebesgue monotone convergence theorem, we can pass the limit as $N\to \infty$ and so complete the proof of Lemma \ref{LEM:Lya}.

\end{proof}

\begin{remark}\rm
\label{REM:MODEN}
We point out that compared to \eqref{VV1}, the modified energy gives us 
\begin{align*}
\frac{d}{dt}  V_{\text{mod}}(\vec L_t(\u_0,\xi) )=-\frac 12 \|  \vec  L_t(\u_0,\xi)(t)    \|_{\H^{\frac 12-}}^2,
\end{align*} 

\noi
which is reflected on \eqref{gen} by exhibiting the additional smoothing effect.
\end{remark}

\appendix
\section{Large-time smoothing estimate and unique ergodicity}
\label{SEC:APP}

In Subsection \ref{SUBSEC:EXPERG}, we already proved that the Gibbs measure is the unique invariant measure under the flow \eqref{SdSG} by obtaining the contraction property \eqref{STARCONTRAC}. In this appendix, we give a discussion regarding a suitable large-time smoothing estimate (see Proposition \ref{PROP:diff}) and so obtain the unique ergodicity,  based on the method introduced in \cite{FT}.  We point out that showing only the unique ergodicity does not rely on the suitable irreducibility condition and Lyapunov structure whose existence were essential to prove the exponential ergodicity result. 

Before we go ahead, we recall the following notations presented in \cite[Section 3.1]{HM06}. The total variation distance between two probability measures is given by the Wasserstein-1 distance $\W_{d_{\text{TV}}}$ corresponding to the metric
\begin{align*}
d_{\text{TV}}(x,y)=\begin{cases}
1 \quad &\text{if} \quad x\neq y\\
0 \quad &\text{if} \quad x=y.
\end{cases}
\end{align*}

\noi
This metric completely separates all points in the phase space, resulting in a total loss of information about the underlying topology. It suggests the following definition, which provides one way of approximating the total variation distance between two probability measures by a sequence of Wasserstein-1 distances.
Let us consider $\{d_{n}\}_{n\in \N}$ to be the following totally separating system of (pseudo)-metrics on $\H^{\frac 12-}$:
\begin{align}
d_{n}(\u_0,\u_1)=1\wedge n \|\u_0-\u_1\|_{\H^{\frac 12-}}. 
\label{dn0}
\end{align}

\noi
Notice that $\lim_{n\to \infty} d_n(\u_0, \u_1)=1$ for all $\u_0 \neq \u_1$. We record the following lemma. For the proof of this lemma, see \cite[Corollary 3.5]{HM06}. 
\begin{lemma}\label{LEM:dnlim}
Let $\{d_n\}_{n\in \N}$ be the totally separating system in \eqref{dn0}. Then, we have
\begin{align*}
\W_{d_{\textup{TV}}}(\mu_1,\mu_2)=\lim_{n\to \infty}  \W_{d_n}(\mu_1, \mu_2)  
\end{align*}

\noi
for any two positive measures $\mu_1$ and $\mu_2$ with equal mass.
\end{lemma}

\begin{remark}\rm
Recall that we denoted the flow of linear equation \eqref{SDLW} by $\vec L_t(\u_0,\xi)$ in \eqref{URISLIN}. Notice that 
\begin{align}
\vec L_t(\u_2^0;\xi)=\vec L_t(\u_1^0;\xi)+\vec S(t)(\u_2^0-\u_1^0).
\label{CURSIS}
\end{align}

\noi
Suppose that $\mu_1$ and $\mu_2$ are invariant measures under the linear flow \eqref{SDLW}. Then, thanks to the invariance of $\mu_1$ and $\mu_2$ and \eqref{CURSIS}, we have 
\begin{align*}
\W_{d_n}(\mu_1,\mu_2)&=\sup_{[F]_{\Lip_{d_n}} \le 1}\bigg(\int F(\u_1^0) d\mu_1(\u_1^0)-\int F(\u_2^0) d\mu_2 (\u_2^0)   \bigg)\\
&\le \int \E\Big[1\wedge n\|\vec S(t)(\u_2^0-\u_1^0)  \|_{\H^{\frac 12-}} \Big] d\mu_1(\u_1^0)d\mu_2 (\u_2^0)\\
&\le \int \E\Big[1\wedge ne^{-\frac t2}\|\u_2^0-\u_1^0  \|_{\H^{\frac 12-}} \Big] d\mu_1(\u_1^0)d\mu_2 (\u_2^0).
\end{align*}

\noi
By  the dominated convergence theorem when taking $t\to \infty$, we have $\W_{d_n}(\mu_1,\mu_2)=0$ for every $n \in \N$. Thanks to Lemma \ref{LEM:dnlim}, we have $\W_{d_{\textup{TV}}}(\mu_1,\mu_2)=0$ and so $\mu_1=\mu_2$. This shows that the linear stochastic flow \eqref{SDLW} has the unique invariant measure\footnote{Regardless of the spatial dimension $d$}.  

\end{remark}

As in \cite{FT}, we now find a control (Girsanov shift) $h$ such that perturbing the noise $\xi+h$ by $h$  control a given perturbation in the initial condition (see Lemma \ref{LEM:Gir}). The existence of a proper control $h$ is to guarantee the large-time smoothing estimate (Proposition \ref{PROP:diff}) and so the unique ergodicity for the Gibbs measure $\rhoo$ under the flow \eqref{SdSG}. Before we prove Proposition \ref{PROP:diff}, we need the following lemma.
\begin{lemma}[Girsanov shift]\label{LEM:Gir}
Let $\u_1^0, \u_2^0 \in \H^{\frac 12-}(\T)$. We define the following process
\begin{align}
h(t;\xi):=\frac 1{\sqrt{2}} \Big( \sin\big( \Phi_t(\u_1^0,\xi)+S(t)(\u_2^0-\u_1^0)    \big) -  \sin\big( \Phi_t(\u_1^0,\xi)    \big) \Big).
\label{shift0}
\end{align}

\noi
Then, $h(t)$ is adapted to the filtration $\{\F_t\}_{t\ge 0}$ induced by $\xi$ and satisfies
\begin{align}
\Phi_t(\u_2^0,\xi+h)=\Phi_t(\u_1^0,\xi)+S(t)(\u_2^0-\u_1^0).
\label{var}
\end{align}

\noi
Moreover, for any $p \ge 1$, we have
\begin{align}
\E\Big[\| h\|^p_{L^2([0,\infty); L_x^2)} \Big] \le \| \u_1^0-\u_2^0 \|_{\H^{\frac 12-}}^p.
\label{shift}
\end{align}

\end{lemma}

\begin{proof}
By exploiting the strucutre of the equation \eqref{SdSG} and the definition of \eqref{shift0}, we can easily obtain the result \eqref{var}. By using the mean value theorem and \eqref{linest}, we have \eqref{shift}.
\end{proof}

We now state the following large-time smoothing estimate.
\begin{proposition}\label{PROP:diff}
Let $n\in \N $, $R>0$, and $\{\Pt{t}: t\ge 0\}$ be the Markov semigroup associated to \eqref{SdSG}. For any given $\u_1^0$, $\u_2^0\in B_R(\0)$, there exist $t_0=t_0(n,R)$, $0<K(R)<1$, and $K'(R)>0$ such that for any $t \ge t_0$, we have
\begin{align}
\big|\Pt{t}F(\u_1^0)-\Pt{t}F(\u_2^0)\big|\le (1-K(R))\big(1 \wedge K'(R)\|\u_1^0-\u_2^0 \|_{\H^{\frac 12-}}^\frac 12\big)
\label{URiS590}
\end{align}

\noi
for any Lipschitz $F$ with $[F]_{\Lip_{d_n}} \le 1$ and $\| F\|_{L^\infty}\le \frac 12$

\end{proposition}

Notice that the asymptotic strong Feller property 
comes as a direct corollary of Proposition \ref{PROP:diff}. We give the proof of Proposition \ref{PROP:diff} at the end of this appendix. We are now ready to prove the unique ergodicity by only assuming Proposition \ref{PROP:diff}. 
\
\begin{proof}[Proof of the unique ergodicity]
Let $\rho_1$ and $\rho_2$ be two distinct ergodic invariant measures. Then, they are mutually singular i.e.~$\rho_1 \perp \rho_2$ (see \cite[Proposition 3.2.5]{DZ}). Fix $n\in \N$. We choose $R>0$ sufficiently large so that $\displaystyle{\min_{j=1,2}}\rho_{j,R}>\frac 12$, where $\rho_{j,R}:=\rho_j(B_R(\0))$. 
Thanks to the Monge–Kantorowitch duality as in \eqref{MGDU1} and the invariance of $\rho_1$ and $\rho_2$, we have that for any $t=t(n,R)>0$ and $K(R)>0$ given by Proposition \ref{PROP:diff},
\begin{equation}
\begin{split}
\W_{d_n}(\rho_1,\rho_2)&=\sup_{\substack{[F]_{\Lip_{d_n}} \le 1\\ \| F\|_{L^\infty}\le \frac 12} }\bigg[\int  F(\u_1^0) d\rho_1(\u_1^0)-\int F(\u_2^0) d\rho_2(\u_2^0) \bigg]\\
&=\sup_{\substack{[F]_{\Lip_{d_n}} \le 1\\ \| F\|_{L^\infty}\le \frac 12}} \iint \E\big[F(\vec \Phi_t(\u_1^0,\xi)) \big]-\E\big[F(\vec \Phi_t(\u_2^0,\xi)) \big] d\rho_1(\u_1^0)d\rho_2(\u_2^0)\\
&\le \sup_{\substack{[F]_{\Lip_{d_n}} \le 1\\ \| F\|_{L^\infty}\le \frac 12}}  \bigg| \iintt_{B_R(0)\times B_R(0)} \E\big[F(\vec \Phi_t(\u_1^0,\xi)) \big]-\E\big[F(\vec \Phi_t(\u_2^0,\xi)) \big] d\rho_1(\u_1^0)d\rho_2(\u_2^0)   \bigg|\\
&\hphantom{XXX}+1-\rho_{1,R}\rho_{2,R}\\
&\le (1-K(R))\rho_{1,R}\rho_{2,R}+1-\rho_{1,R}\rho_{2,R}\le 1-\frac 14K(R).
\end{split}
\label{dn}
\end{equation}

\noi
Since the right hand side of \eqref{dn} does not depend on $n \in N$, by taking $n\to \infty$ and Lemma \ref{LEM:dnlim}, we obtain 
\begin{align*}
\W_{d_{\textup{TV}}}(\rho_1,\rho_2)\le 1-\frac 14K(R)<1.
\end{align*}

\noi
This contradicts the fact that $\rho_1$ and $\rho_2$ are distinct ergodic invariant measures (and so $\rho_1\perp\rho_2$). Hence, there exists only one ergodic invariant measure.
By the ergodic decomposition theorem (see \cite[Theorem 5.1]{Hairlec}), every invariant measure is a convex combination of ergodic measures. Therefore, \eqref{SdSG} has the unique invariant measure.
\end{proof}

Before we present the proof of Proposition \ref{PROP:diff}, we introduce the following lemma whose proof is presented at \cite[Lemma 5.11]{FT}.
\begin{lemma}\label{LEM:EXURIS88}
Let $1\le p <\infty$ and $X$ be a random variable so that $\E\big[e^{X}\big]=1$ and $\E\big[|X|^p\big]<\infty$. Then, for any $L>0$, we have
\begin{align}
\E\big[ |e^{X}-1| \big]\le 2\bigg(1-e^{-L}+e^{-L}L^{-p}\E\big[|X|^p\big]\bigg).
\label{exp}
\end{align}

\end{lemma}

We are now ready to prove Proposition \ref{PROP:diff}.
\begin{proof}[Proof of Proposition \ref{PROP:diff}]
From \eqref{var}, we have
\begin{align}
&\E\big[F(\vec \Phi_t(\u_1^0,\xi))\big]-\E\big[F(\vec \Phi_t(\u_2^0,\xi))\big]\notag \\
&=\bigg(\E\big[F\big(\vec \Phi_t(\u_2^0,\xi+h)-\vec S(t)(\u_2^0-\u_1^0) \big) \big]-\E\big[ F(\vec \Phi_t(\u_2^0,\xi+h ) )       \big] \bigg) \notag \\
&\hphantom{X} +\bigg( \E\big[ F(\vec \Phi_t(\u_2^0,\xi+h ) ) -\E\big[F(\vec \Phi_t(\u_2^0,\xi))\big]      \big] \bigg)=\text{I}+\II.
\label{dec1}
\end{align}

\noi
We first consider the first term $\text{I}$ in \eqref{dec1}. Since $F\in \Lip_{d_n}$, we have
\begin{align}
\text{I}&\le [ F]_{\Lip_{d_n}}\E\big[1\wedge n\| \vec S(t)(\u_2^0-\u_1^0)\|_{\H^{\frac 12-}} \big]\le [F ]_{\Lip_{d_n}}\big(1\wedge 2ne^{-\frac t2}\|\u_2^0-\u_1^0 \|_{\H^{\frac 12-} }\big).
\label{Lip}
\end{align}

\noi
Next, we consider the second term $\II$ in \eqref{dec1}. We set the stochastic exponential of $h(t)$ in \eqref{shift0} as follows
\begin{align*}
\EE(h(t)):=\exp\bigg\{ -\frac 12 \int_0^t \| h(t')\|_{L^2_x}^2 dt'+\int_0^t \jb{h(t'),d\xi(t')}_{L^2_x} \bigg \}.
\end{align*}

\noi
Notice that $h(t)$ satisfies the Novikov condition and so $d\mathbb{Q}=\EE(h(t))d\PP$ is a probability measure. It follows from Girsanov's theorem that 
\begin{align}
\II=\E\big[ F(\vec \Phi_t(\u_2^0,\xi))( \EE(h(t))-1 ) \big] \le \| F\|_{L^\infty } \E\big[  |\EE(h(t))-1| \big].
\label{Gir}
\end{align}

\noi
Thanks to Lemma \ref{LEM:EXURIS88}, It\^o isometry, and \eqref{shift},  we have 
\begin{align}
\eqref{Gir}&\le 2 \| F \|_{L^\infty } \bigg[ 1-e^{-L}+e^{-L}L^{-1}\E \big[ \| h\|^2_{L_{t,x}^2} +\| h\|_{L_{t,x}^2}  \big]  \bigg] \notag \\
& \le 2 \| F\|_{L^\infty }\bigg[1-e^{-L}\Big(1-L^{-1}\max\big\{\|\u_1^0-\u_2^0 \|_{\H^{\frac 12-} }, \|\u_1^0-\u_2^0 \|_{\H^{\frac 12-} }^2 \big\} \Big)  \bigg],
\label{Gir2}
\end{align}

\noi
where $L\gg 1$ will be specified later.
Then, from \eqref{dec1}, \eqref{Lip},\eqref{Gir}, and \eqref{Gir2}, we obtain 
\begin{align}
&\E\big[F(\vec \Phi_t(\u_1^0,\xi))\big]-\E\big[F(\vec \Phi_t(\u_2^0,\xi))\big] \notag \\
&\le [ F]_{\Lip_{d_n}}\big(1\wedge 2ne^{-\frac t2}\|\u_2^0-\u_1^0 \|_{\H^{\frac 12-} }\big)\notag\\
&\hphantom{X}+2 \| F\|_{L^\infty }\bigg[1-e^{-L}\Big(1-L^{-1}\max\big\{\|\u_1^0-\u_2^0 \|_{\H^{\frac 12-} }, \|\u_1^0-\u_2^0 \|_{\H^{\frac 12-} }^2 \big\} \Big)  \bigg].
\label{e1}
\end{align}

\noi
We first consider the second term in \eqref{e1}. By choosing $L=L(R)$ sufficiently large so that $\max\{R,R^2\}L^{-1}\ll 1 $, we have
\begin{align}
1-e^{-L}\Big(1-L^{-1}\max\big\{\|\u_1^0-\u_2^0 \|_{\H^{\frac 12-} }, \|\u_1^0-\u_2^0 \|_{\H^{\frac 12-} }^2 \big\} \Big)   \le 1-\frac 18e^{-L(R)}.
\label{e2}
\end{align}

\noi
Next, by choosing $t_R:=t_R(n,R,L(R))$ sufficiently large, we obtain
\begin{align}
\big(1\wedge 2ne^{-\frac 12  t_R}\|\u_2^0-\u_1^0 \|_{\H^\alpha}\big)\le 2\big(1\wedge 2nRe^{-\frac 12t_R} \big)\le \frac 1{16} e^{-L(R)}.
\label{e3}
\end{align}

\noi
Hence, by combining \eqref{e1}, \eqref{e2}, and \eqref{e3}, we have that for any $t \ge t_R$,
\begin{align}
\sup_{\substack{[F]_{\Lip_{d_n}} \le 1\\ \| F\|_{L^\infty}\le \frac 12} }\bigg|\E\big[F(\vec \Phi_t(\u_1^0,\xi))\big]-\E\big[F(\vec \Phi_t(\u_2^0,\xi))\big] \bigg|\le 1-\frac 1{16}e^{-L(R)}:=1-K(R),
\label{c1}
\end{align}

\noi
which implies the first part in \eqref{URiS590}. 

We now consider the case $\|\u_1^0-\u_2^0\|_{\H^\frac 12-} \ll 1$ to obtain the second part in \eqref{URiS590}. By choosing $t_1=t_1(n) \gg 1$, we have
\begin{align}
1\wedge 2ne^{-\frac t2}\|\u_2^0-\u_1^0 \|_{\H^{\frac 12-} }\le \|  \u_2^0-\u_1^0   \|_{\H^{\frac 12-} } \le \|  \u_2^0-\u_1^0 \|_{\H^{\frac 12-} }^\frac 12.
\label{MULchi0}
\end{align}

\noi
for any $t\ge t_1$. We choose $L=R^{-\frac 12}\|\u_1^0-\u_2^0 \|_{\H^{\frac 12-}}^\frac 12$. Then, we have $e^{-L}\ge 1-L=1-R^{-\frac 12}\|\u_1^0-\u_2^0 \|_{\H^{\frac 12-}}^\frac 12$. Therefore, from \eqref{e1} and \eqref{MULchi0}, we obtain that for any $t\ge t_1$,
\begin{align}
\sup_{\substack{[F]_{\Lip_{d_n}} \le 1\\ \| F\|_{L^\infty}\le \frac 12} }\bigg|\E\big[F(\vec \Phi_t(\u_1^0,\xi))\big]-\E\big[F(\vec \Phi_t(\u_2^0,\xi))\big] \bigg|&\le (1+R^{-\frac 12}+e^{-L}R^{\frac 12}) \|  \u_2^0-\u_1^0 \|_{\H^{\frac 12-}}^\frac 12 \notag \\
&\le C'(R)\|  \u_2^0-\u_1^0 \|_{\H^{\frac 12-}}^\frac 12.
\label{c2}
\end{align}

\noi
By combining \eqref{c1} and \eqref{c2}, we conclude that there exists $t_0=t_0(n,R)$ such that for any $t \ge t_0$,
\begin{align*}
\sup_{\substack{[F]_{\Lip_{d_n}} \le 1\\ \| F\|_{L^\infty}\le \frac 12} }\bigg|\E\big[F(\vec \Phi_t(\u_1^0,\xi))\big]-\E\big[F(\vec \Phi_t(\u_2^0,\xi))\big] \bigg| \le  (1-K(R))\big(1 \wedge K'(R)\|\u_1^0-\u_2^0 \|_{\H^{\frac 12-}}^\frac 12\big),
\end{align*}

\noi
where $K'(R)=\frac{C'(R)}{1-K(R)}$. Therefore, we conclude the proof of Proposition \ref{PROP:diff}. 

\end{proof}

\begin{ackno}\rm
The author would like to thank Pavlos Tsatsoulis for many helpful discussions.
\end{ackno}

\end{document}